%% file: main.tex
\documentclass[12pt]{article}

\usepackage{setspace}
\usepackage[english]{babel}
\usepackage[utf8]{inputenc}
\usepackage{amsmath}
\usepackage{scalerel}
\usepackage{amssymb}
\usepackage{parskip}
\usepackage{graphicx}
\usepackage{textcomp}
\usepackage{fancyhdr}
\usepackage{amsthm}
\usepackage{amssymb}
\usepackage{graphicx}
\usepackage{appendix}
\usepackage{indentfirst}
\usepackage{mathtools}
\usepackage[noadjust]{cite}
\usepackage{bm}

\usepackage[framemethod=tikz]{mdframed}
\usepackage[top=2.5cm, left=3cm, right=3cm, bottom=4.0cm]{geometry}
\allowdisplaybreaks
\numberwithin{equation}{section}
\newcommand{\kp}{{\frac{k}{n}+pT^\e_k}}
\newcommand{\kpp}{{\frac{k}{n}+(p+1)T^\e_k}}

\newcommand{\given}{\mid\mathcal F_{\kk}}
\newcommand{\rcoe}{\nabla H(X_s^\e)^*\sigma(X_s^\e)}
\newcommand{\e}{\varepsilon}
\newcommand{\h}{\hat}
\newcommand{\kk}{\frac{k}{n}}

\newcommand{\kl}{\frac{k+1}{n}}
\newcommand{\pl}{\varphi^{(l)}}

\usepackage[active]{srcltx}

\newtheorem{theorem}{Theorem}[section]

\newtheorem{lemma}[theorem]{Lemma}

\newtheorem{definition}[theorem]{Definition}

\title{Large Deviations for Hamiltonian Systems on Intermediate Time Scales}
\author{Shuo Yan\\
Department of Mathematics, University of Maryland\\ 4176 Campus Drive - William E. Kirwan Hall\\
College Park, Maryland 20742-4015,
United States\\
shuoyan@umd.edu}
\date{}

\begin{document}

\maketitle

\input{abstract}

\input{Introduction}

\input{Mainresult}

\input{LowerBound}

\input{Upperbound}

\input{ExponentialTighiness}

\input{proof}

\input{appendix}

\input{Acknowledgement}

\end{document}

%% file: abstract.tex
\begin{abstract}
    We consider a two-dimensional Hamiltonian system perturbed by a small diffusion term, whose coefficient is state-dependent and non-degenerate. 
    As a result, the process consists of the fast motion along the level curves and slow motion across them. 
    On finite time intervals, the large deviation principle applies, while on time scales that are inversely proportional to the size of the perturbation, the averaging principle holds, i.e., the projection of the process onto the Reeb graph converges to a Markov process.
    In our paper, we consider the intermediate time scales and prove the large deviation principle, with the action functional determined in terms of the averaged process on the graph.
    

\end{abstract}

%% file: Introduction.tex
\section{Introduction}
Consider the Hamiltonian dynamical system in $\mathbb R^2$ defined by an ordinary differential equation:
\begin{equation}
\label{eq:1}
    dx_t=v(x_t)dt,\ x_0\in\mathbb R^2,
\end{equation}
where
$$v(x)=\nabla^\perp H(x):=\left(-\frac{\partial H(p,q)}{\partial q},\frac{\partial H(p,q)}{\partial p}\right),$$ and Hamiltonian $H$ is smooth enough. 
\begin{figure}[!ht] 
    \centering
    \includegraphics[width=0.95\textwidth]{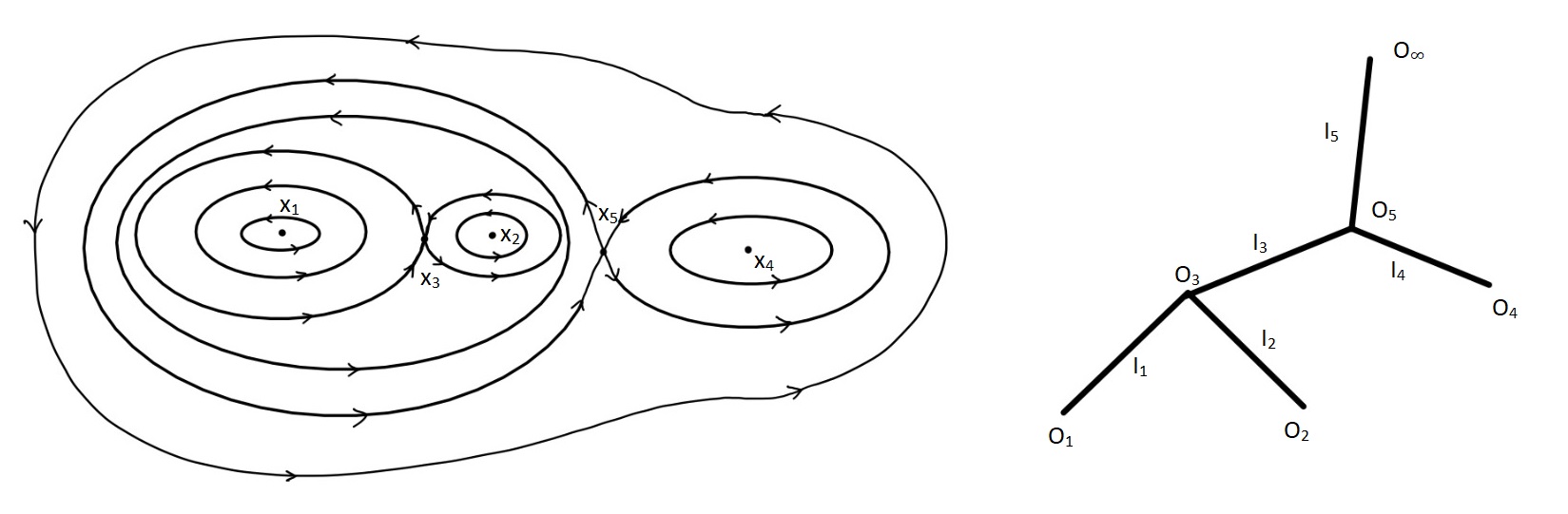}
    \caption{A typical example of a Hamiltonian system and the corresponding Reeb graph.}
    \label{hamil}
\end{figure}

Small random perturbations of Hamiltonian systems have been studied from different perspectives. 
In particular, let $\e$ be a small positive parameter, and $\Tilde X^\e$ be a diffusion process in $\mathbb R^2$ defined by the equation:
\begin{equation}
    \label{eq:2}
    d\Tilde X_t^\e=\nabla^{\perp}H(\Tilde X_t^\e)dt+\sqrt\e\sigma(\Tilde X_t^\e)d\Tilde{W_t},\ \ \ \ \ \ \Tilde{X_0^\e}=x_0\in\mathbb R^2,
\end{equation}where $H$ is a smooth function from $\mathbb R^2$ to $\mathbb R$ with bounded second derivatives; $\sigma(x)$ is a smooth $2\times l$ matrix-valued function such that $\sigma(x)\sigma^*(x)$ is uniformly positive-definite and bounded;
$\Tilde{W_t}$ is a standard $l-$dimensional Wiener process. 
On time scales of order one, the large  deviation principle, known as Freidlin-Wentzell theory, established in \cite{Ventsel1970}, describes the probability (exponentially small in $1/\e$) of the event that a realization of 
$\Tilde X_t^\e$ belongs to a neighborhood of a path that is different from the trajectory of the deterministic flow \eqref{eq:1}. 
On time scales of order $1/\e$, the averaging principle tells us that the projection of the process $\Tilde X_t^\e$ onto the Reeb graph behaves as a strong Markov process. In particular, on each edge of the graph, $H(\Tilde X_t^\e)$ behaves as a diffusion process (\cite{fred}). 
On intermediate time scales, i.e., when time is of order $|\log\e|\e^{\beta-1}$, a result of the averaging type can also be established (\cite{Hairer2016AFK}). 
Namely the time-changed process converges to a motion on a rescaled graph, provided it starts on or near a level curve containing a saddle point of $H$. 
In this paper, we also study the intermediate time scales, $t\sim\e^{\beta-1}$, however, without spatial rescaling. 
Our main result states that, after projection onto the graph, the process satisfies the large deviation principle on time scales of order $\e^{\beta-1}$. 
The action functional can be described in terms of the coefficients of the
averaged process on the graph. 
In the particular case $\beta=1/2$, estimates on transition probability for the process inside the edges and estimates for the exit time from a neighborhood of an interior vertex were given in \cite{MR1883738}.
(The main focus of the latter paper was on the propagation of the reaction front in the KPP equation with the underlying diffusion given by \eqref{eq:2}.)
In our paper, we allow arbitrary $\beta\in(0,1)$, which requires a more complicated approach.

Since we are interested in the behavior of $\Tilde X_t^\e$ with $t\sim \e^{\beta-1}$, it is convenient to rescale the time by defining $X_t^\e=\Tilde X_{t\e^{\beta-1}}^\e.$ Then we have the equation for $X_t^\e$:
\begin{equation}
\label{eq:3}
    dX_t^\e=\e^{\beta-1}\nabla^{\perp}H(X_t^\e)dt+\e^{\beta/2}\sigma(X_t^\e)dW_t,\ \ \ \ \ \ X_0^\e=x_0,
\end{equation}
where $W_t$ is an $l-$dimensional Wiener process. By Ito's formula applied to $H(X_t^\e)$,
\begin{equation}
\label{eq:4}
    H(X_t^\e)=H(x_0)+\e^\beta\int_0^t AH(X_s^\e)ds+\e^{\beta/2}\int_0^t\rcoe dW_s,
\end{equation}
where $A$ is the operator $Au(x)=\frac{1}{2}\sum_{i,j}[\sigma(x)\sigma^*(x)]_{i,j}\cdot\frac{\partial^2}{\partial x_i\partial x_j}u(x)$. 
Note that $AH(x)$ is uniformly bounded, so, due to the $\e^\beta$ factor, the second term on the right hand side does not contribute to the action functional. 
However, this term is important to understand the behavior of the process near the critical points of $H$, where the integrand in the last term  in \eqref{eq:4} vanishes. 
In order to understand large deviations, we should study how the coefficients are averaged along the trajectories of the perturbed process. 

Let $\Gamma$ denote the Reeb graph corresponding to the Hamiltonian $H$, and let $Y: \mathbb{R}^2 \to \Gamma$ be the projection on the graph. 
Let $\bm{\mathrm x}_1, ..., \bm{\mathrm x}_N$ be all the critical points of $H$, and $O_k=Y(\bm{\mathrm x}_k)$, $k=1, ... , N$. 
Let the edges of $\Gamma$ be labeled as $I_1, ... , I_{N'}$ and a symbol $\sim$ between a vertex and an edge means that the vertex is an endpoint of the edge. Then define:
\begin{itemize}
\item$D_i$ is the set of all points $x\in\mathbb R^2$ such that $Y(x)$ belongs to the interior of $I_i$;
\item$C_k=\{x:Y(x)=O_k\}$ is the extremum point $\bm{\mathrm x}_k$ or the separatrix containing $\bm{\mathrm x}_k$; 
\item$C_i(H)=\{x\in D_i:H(x)=H\}$ is one of the connected components of the level set of $H$;
\item$D_i(H_1,H_2)=\{x\in D_i:H_1<H(x)<H_2\}$, provided that $H_1<H_2$, is the set between $C_i(H_1)$ and $C_i(H_2)$;
\item$D_k(\pm\delta)$ is the connected component of the set $\{H(\bm{\mathrm  x}_k)-\delta<x<H(\bm{\mathrm  x}_k)+\delta\}$ containing $C_k$;
\item$D(\pm\delta)=\bigcup_{k}D_k(\pm\delta)$;
\item$C_{ki}(\delta)=\{x\in D_i:H(x)=H(\bm{\mathrm  x}_k)\pm\delta\}$, for $O_k\sim I_i$;
\item$T_i(H)=\oint_{C_i(H)}\frac{1}{|\nabla H(x)|}dl$, for applicable $H$, is the rotation time of system \eqref{eq:1};
\item$B^2_i(H)=\frac{1}{T_i(H)}\oint_{C_i(H)}\frac{|\nabla H(x)^*\sigma(x)|^2}{|\nabla H(x)|}dl$ for $H$ in $ H(I_i)$; $B^2_i(H)=0$ for $H=H(\bm{\mathrm  x}_k)$ and $O_k\sim I_i$, is the ``averaged'' diffusion coefficient w.r.t. the invariant measure on $C_i(H)$.
\end{itemize}

It's worth noting that $T_i(H)$ and $B_i(H)$ have nice regularity properties if $H$ is smooth enough. The following lemma is a direct application of Lemma 8.1.1 from \cite{fred}.
\begin{lemma}
\label{lipschitz}
 $T_i(H)$ and $B^2_i(H)$ are $k-1$ times continuously differentiable at the interior points of the interval $H(I_i)$ if $H$ is $k$ times continuously differentiable. (This implies the Lipschitz continuity of $T_i(H)$, $B^2_i(H)$ and $T_i(H)B_i^2(H)$ on any closed interval $I\subset H(I_i)$ if $k\geq 2$.)
\end{lemma}
In Section 2, we state the necessary assumptions and formulate the main result. 
The proof comes in the sections that follow. 
The lower and upper bounds of the local large deviation principle are proved in Sections 3 and 4, respectively. 
Then, exponential tightness is justified in Section 5. 
These arguments, together with Puhalskii's Theorem (cf. \cite{semi} and \cite{Puhalskii}), imply the main result. 
Some technical proofs are given in the Appendix.

%% file: Mainresult.tex
\section{Main Results}
Our main result concerns the large  deviation principle for the process $Y(\Tilde X_{t\e^{\beta-1}}^\e)$. Large deviations from averaged behavior are described in terms of the action functional, which (intuitively) gives the probability that the trajectories of the process are close to a given deterministic trajectory. More precisely, we give the formal definition of an action functional (\cite{fred}): \begin{definition}
    Let $(X,\rho)$ be a metric space. Let $\mu^\e$, $\e>0$, be a family of probability measures on $\mathcal B(X)$, $\lambda(\e)$ be a positive real-valued function going to $+\infty$ as $\e\downarrow0$, and $S(x)$ be a function on $X$ with values in $[0,\infty]$. It is said that $\lambda(\e)S(x)$ is an action function (action functional if $X$ is a function space) for $\mu^\e$ as $\e\downarrow0$ if:
\begin{enumerate}
\item[(0)]the set $\Phi(s)=\{x:S(x)\leq s\}$ is compact for every $s\geq0$;
\item[(1)] for each $\delta>0$, each $\gamma>0$, and each $x\in X$, there exists $\e_0>0$ such that$$\mu^\e\{y:\rho(x,y)<\delta\}\geq\exp\{-\lambda(\e)[S(x)+\gamma]\}$$ for all $\e\leq \e_0$;
\item[(2)] for each $\delta>0$, each $\gamma>0$, and each $s>0$, there exists $\e_0>0$ such that $$\mu^\e\{y:\rho(y,\Phi(s))\geq\delta\}\leq\exp\{-\lambda(\e)(s-\gamma)\}$$ for all $\e\leq \e_0$.
\end{enumerate}
\end{definition}

Consider the following metric on $\Gamma$: $r(y_1,y_2)$ is the length of the shortest path connecting $y_1$ and $y_2$. For example, if $y_1=(i,H_1)$, $I_1\sim O_1$, $O_1\sim I_2$, $I_2\sim O_2$, $O_2\sim I_3$ and $y_2=(3,H_2)$, then $r(y_1,y_2)=|H_1-H(\bm{\mathrm  x}_1)|+|H(\bm{\mathrm  x}_1)-H(\bm{\mathrm  x}_2)|+|H(\bm{\mathrm  x}_2)-H_2|$. Based on this metric, we introduce the uniform metric on $\textbf{C}\left([0,T],\Gamma\right)$: $\rho_{0,T}(\bm\varphi,\bm\psi)=\sup_{0\leq t\leq T}r(\bm\varphi(t),\bm\psi(t))$. In the metric space $(\textbf{C}\left([0,T],\Gamma\right),\rho_{[0,T]})$, we consider the probability measures induced by the family of process $Y(\Tilde X_{t\e^{\beta-1}}^\e)$. 
We use the conventions $0/0=0$, $\sup\{\emptyset\}=-\infty$, and $\inf\{\emptyset\}=+\infty$. 
A trajectory $\bm\varphi$ on $\Gamma$ can be written in terms of its components as $(i_t,\varphi_t)$.
Define the functional $S(\bm\varphi)$ on $\textbf{C}\left([0,T],\Gamma\right)$: $$S(\bm\varphi)=\frac{1}{2}\int_0^T\frac{|\dot\varphi_t|^2}{B_{i_t}^2(\varphi_t)}dt$$for $\varphi\in\textbf{C}\left([0,T],\mathbb R\right)$ that is absolutely continuous and satisfies $\dot\varphi_t/B_{i_t}(\varphi_t)\in L^2([0,T])$; otherwise $S(\bm\varphi)=\infty$. 
Note that if $\bm\varphi_t=O_k$ for some interior vertex $O_k$, then we have different representations of $\bm\varphi_t$ since it's on multiple edges at the same time. 
However, by the definition of $B^2$, there's no ambiguity because, according to our conventions, the integrand is defined as $0$ at such vertices (see Lemma~\ref{zeroder}).
The following conditions are assumed to hold throughout the paper: 
\begin{enumerate}
\item The Hamiltonian $H(x)$, $x\in\mathbb R^2$, is four times continuously differentiable with bounded second derivatives.
\item For sufficiently large $|x|$, there exist positive $A_1$, $A_2$ and $A_3$, such that $H(x)\geq A_1|x|^2$, $|\nabla H(x)|\geq A_2|x|$, and $\Delta H(x)\geq A_3$.
\item $H(x)$ has a finite number of critical points, and the matrix of second derivatives is nondegenerate at those points.
\item Each level curve corresponding to a vertex on the Reeb graph contains at most one critical point.
\item$\sigma(x)$ is continuously differentiable with bounded derivatives and $\sigma(x)\sigma^*(x)$ is bounded and uniformly positive-definite.
\end{enumerate}

Now we are ready to formulate the main result:
\begin{theorem}
\label{main}
    Let the assumptions above be satisfied and let the process $\Tilde X_t^\e$ be defined as in \eqref{eq:2}. Then $\e^{-\beta}S(\bm\varphi)$ is the action functional of the family of process $Y(\Tilde X_{t\e^{\beta-1}}^\e)$ in the sense of the uniform metric on $\textbf{C}([0,T],\Gamma)$ restricted to the set of functions that start at $Y(x_0)$.
\end{theorem}

%% file: LowerBound.tex
\section{Lower bound for the Large Deviation Principle}
Let us prove the lower bound for the large deviation principle. First, we look at a process that starts and evolves inside a single edge. Next, we deal with a process that starts at an exterior vertex but stays within an edge. Finally, we estimate the probability for a process to travel through an interior vertex.

\subsection{The case where the process evolves inside one edge}
In this section, the process evolves only inside one edge (for example $I_i$). Therefore we simplify the notations when there is no ambiguity: $T_i(H)\to T(H)$; $B_i^2(H)\to B^2(H)$; $S(\bm{\varphi})=S(i,\varphi)\to S(\varphi)$, etc. The following lemma tells us that if the process starts sufficiently near the initial point of the the deterministic trajectory, it has a good chance to follow the trajectory and arrive in a tiny neighborhood of where the trajectory ends. 
\begin{lemma}
\label{loclem}
For every $\delta>0$, $\gamma>0$, and $\bm{\varphi}_t=(i,\varphi_t)\in\textbf{C}\left([0,T], I_i\right)$ such that $\inf_{t\in[0,T]}r(\varphi_t,H(O))>0$ for each $O\sim I_i$, there exists $h>0$ such that for every $h'>0$ there exists $\e_0>0$ such that for every initial point $X^\e_0=x$ satisfying $|H(x)-\varphi_0|<h$,
$${\bm{\mathrm P}}(|H(X_T^\e)-\varphi_T|<h',\ \rho_{0,T}(H(X_t^\e),\varphi_t)<\delta)\geq\exp\left[-\e^{-\beta}(S(\varphi)+\gamma)\right],$$for all $\e<\e_0$.
\label{lowerbound}
\end{lemma}
\begin{proof}
We present the proof in several steps. \textbf{\romannumeral 1.} We can assume that $h'\leq h$ and that $\varphi$ is absolutely continuous with $\varphi'\in L^2([0,T])$ because $S(\varphi)=\infty$ for any other $\varphi$. 
Also, without loss of generality, we assume that $T=1$, $O_j\sim I$, $O_k\sim I$, $H(\bm{\mathrm  x}_j)<H(\bm{\mathrm  x}_k)$, and $H(\bm{\mathrm  x}_j)=0$. 
We can assume that $0<\Tilde m\leq\varphi\leq\Tilde M$ and $\delta<\frac{\Tilde m}{2}\wedge\frac{H(\bm{\mathrm  x}_k)-\Tilde M}{2}$. 
Let $C=\overline{D_i(\Tilde m-\delta,\Tilde M+\delta)}$, which is a compact set in $\mathbb R^2$.
Then there exist $m>0$, $M>0$, $L>0$ such that:
\begin{enumerate}
\item[(1)] $\lvert AH(x)\rvert+\lvert\sigma(x)\rvert<M,\ \forall x\in\mathbb R^2;\ |\nabla H(x)^*\sigma(x)|^2<M,\ \forall x\in C$;
\item[(2)] $|\nabla H(x)^*\sigma(x)|^2>m,\ \forall x\in C$;
\item[(3)] $m<T_i(H)<M\text{  and  }m<B_i^2(H)<M,\ \forall H\in\{H:r(H,Ran(\varphi))<\delta\}$;
\item[(4)] $\lvert\nabla H(x)-\nabla H(y)\rvert+\left|\lvert\nabla H(x)^*\sigma(x)\rvert^2-\lvert\nabla H(y)^*\sigma(y)\rvert^2\right|<L\lvert x-y\rvert,\ \forall x,y\in C$;
\item[(5)] $\lvert T_i(H_1)-T_i(H_2)\rvert+\lvert T_i(H_1)B_i^2(H_1)-T_i(H_2)B_i^2(H_2)\rvert<L\lvert H_1-H_2\rvert,\ \forall H_1,H_2\in\{H:r(H,Ran(\varphi))<\delta\}$ (see Lemma~\ref{lipschitz}).
\end{enumerate}
With $m,M,L$ selected, we choose the parameters as follows:
\begin{enumerate}
\item[(1)]$\delta'>0$ such that $\delta'<\delta$, $\alpha:=\frac{8L(M+1)}{m^2}\delta'<1/2$ and $\frac{\alpha}{1-\alpha}\int_0^1 \frac{|\varphi'_t|^2}{B^2(\varphi)}dt<\gamma$;
\item[(2)]$r>0$ such that $r<\frac{m}{24M}$;
\item[(3)]$n\in\mathbb N$ such that $a,b\in[0,T]$ with $|a-b|<1/n$ implies $|\varphi_a-\varphi_b|<r\delta'<\frac{1}{12}\delta'$;
\item[(4)]$h>0$ such that $h<r\delta'$ and $\frac{\alpha}{1-\alpha}\int_0^1 \frac{|\varphi'_t|^2}{B^2(\varphi)}dt+\frac{1}{1-\alpha}\cdot\frac{n^2h\delta'}{m^2}<\gamma$
\end{enumerate}

\textbf{\romannumeral2.} Now let us study how $X_t^\e$ behaves on one interval $[\kk,\kl]$. To start with, $X_t^\e$ satisfies the equation: $$dX_t^\e=\e^{\beta-1}\nabla^{\perp}H(X_t^\e)+\e^{\beta/2}\sigma(X_t^\e)dW_t,\ \ \ \ \ \ X_0^\e=x.$$
For every $0\leq k<n$, define $A_k$ as the event where the process is not perturbed significantly by the diffusion during the time required to make one rotation (which will allow us to use averaging) and does not travel far from the original level curve. More precisely,
\begin{align*}
A_k=&\bigcap_{p=0}^{\lfloor\frac{1}{nT^\e_k}\rfloor}\left\{\sup_{pT^\e_k+\kk\leq t\leq (p+1)T^\e_k+\kk}\e^{\beta/2}\left|\int_{k/n+pT^\e_k}^t \sigma(X_s^\e)dW_s\right|\leq M\cdot\e^{\frac{1-\beta}{4}}\right\}\\
&\ \bigcap\ \ \left\{\sup_{0\leq t\leq\frac{1}{n}}\left|H(X_{\kk+t}^\e)-H(X_\kk^\e)\right|\leq\frac{\delta'}{2}\right\},
\end{align*}where $T^\e_k=\e^{1-\beta}T(H(X_\kk^\e))$. 
By Lemma \ref{ak}, for the initial point $X_k^\e=x$ satisfying $|H(x)-\varphi_\kk|<h$,
$${\bm{\mathrm P}}\left(\Omega\setminus A_k\right)\leq\exp\left(-\frac{n{\delta'}^2}{64M\e^{\beta}}\right)=:\lambda_3(\e).$$
For each $0\leq p\leq\lfloor\frac{1}{nT^\e_k}\rfloor$ and $t\geq\kk+pT^\e_k$, define a deterministic process $\xi_t^\e$ with random initial data by the equation:
$$d\xi_t^\e=\e^{\beta-1}\nabla^\perp H(\xi_t^\e)dt,\ \ \ \ \xi_{\kk+pT^\e_k}^\e=X^\e_{\kk+pT^\e_k}.$$
On the event $A_k\cap\left\{\lvert\varphi_\kk-H(X_\kk^\e)\rvert<h\right\}\subset \{X_t^\e,\xi_t^\e\in C,\ t\in[\kk,\kl]\}$, estimate the difference $\lvert X_t^\e-\xi_t^\e\rvert$ on the interval $[\kp,\kpp]$:
\begin{align*}
    \lvert X_t^\e-\xi_t^\e\rvert\leq\ &\e^{\beta-1}\int_{k/n+pT^\e_k}^t\left|\nabla H^\perp(X_s^\e)-\nabla H^\perp(\xi_s^\e)\right|ds+\e^{\beta/2}\left|\int_{k/n+pT^\e_k}^t\sigma(X_s^\e)dW_s\right|\\
    \leq\ &\e^{\beta-1}\int_{k/n+pT^\e_k}^tL\left|X_s^\e-\xi_s^\e\right|ds+M\cdot\e^{\frac{1-\beta}{4}}.
\end{align*}
So, for $\e$ small enough, by the Gronwall's inequality,
\begin{equation}
    \lvert X_t^\e-\xi_t^\e\rvert\leq M\exp(L\cdot T(H(X_\kk^\e)))\cdot\e^{\frac{1-\beta}{4}}\leq M\exp(LM)\cdot\e^{\frac{1-\beta}{4}}\leq\frac{\delta'}{M}.
    \label{eq:4.1}
\end{equation}
We approximate $\lvert\rcoe\rvert^2$ with $B^2(\varphi_s)$ on the interval $[\kp,\kpp]$, still on the event $A_k\cap\left\{\lvert\varphi_\kk-H(X_\kk^\e)\rvert<h\right\}$:
\begin{align*}
    &\int_\kp^\kpp\lvert\rcoe\rvert^2ds\\ 
    \leq &\int_\kp^\kpp\lvert\nabla H(\xi_s^\e)^*\sigma(\xi_s^\e)\rvert^2ds+LT^\e_k\cdot\frac{\delta'}{M}\\
    \leq&\int_\kp^{\frac{k}{n}+pT^\e_k+\e^{1-\beta}T(H(X_\kp^\e))}\lvert\nabla H(\xi_s^\e)^*\sigma(\xi_s^\e)\rvert^2ds+\e^{1-\beta}\cdot L\delta'M+\e^{1-\beta}T(H(X_\kk^\e))L\cdot\frac{\delta'}{M}\\
    \leq&\ \e^{1-\beta}T(H(X_\kp^\e))B^2(H(X_\kp^\e))+\e^{1-\beta}\cdot L\delta'(M+1)\\
    \leq&\ \e^{1-\beta}T(\varphi_\kp)B^2(\varphi_\kp)+\e^{1-\beta}\cdot L\delta'(M+2)\\
    \leq&\ \e^{1-\beta}T(H(X_\kk^\e))B^2(\varphi_\kp)+\e^{1-\beta}\cdot L\delta'(2M+2)\\
    =&\ T^\e_k B^2(\varphi_\kp)+2\e^{1-\beta}\cdot L\delta'(M+1),
\end{align*}
where the first inequality is by \eqref{eq:4.1}, the second inequality is by Lipschitz continuity of $T(H)$ and boundedness of $\lvert\nabla H(\xi_s^\e)^*\sigma(\xi_s^\e)\rvert$, and the fourth inequality is by the Lipschitz continuity of $T(H)B^2(H)$.

With the same reasoning, we get an estimate in the other direction:
$$\int_\kp^\kpp\lvert\rcoe\rvert^2ds\geq T^\e_k B^2(\varphi_\kp)-2\e^{1-\beta}\cdot L\delta'(M+1).$$
Note that these are valid for every $0\leq p\leq\lfloor\frac{1}{nT^\e_k}\rfloor$. Therefore, on the event $A_k\cap\left\{\lvert\varphi_\kk-H(X_\kk^\e)\rvert<h\right\}$, $\int_\kk^\kl|\rcoe|^2$ can also be approximated by $\int_\kk^\kl B^2(\varphi_s)$ on the interval $[\kk,\kl]$,
\begin{align*}
    &\int_\kk^\kl\lvert\rcoe\rvert^2ds\\
    \leq\ &\sum_{p=0}^{\lfloor\frac{1}{nT^\e_k}\rfloor-1}\int_\kp^\kpp\lvert\rcoe\rvert^2ds+T^\e_k M\\
    \leq\ &\sum_{p=0}^{\lfloor\frac{1}{nT^\e_k}\rfloor-1}T^\e_k B^2(\varphi_\kp)+\frac{1}{nT^\e_k}\cdot2\e^{1-\beta}\cdot L\delta'(M+1)+T^\e_k M\\
    \leq\ &(1+\alpha/2)\int_\kk^\kl B^2(\varphi_s)ds+\frac{4L(M+1)}{mn}\cdot\delta'\\
    \leq\ &(1+\alpha)\int_\kk^\kl B^2(\varphi_s)ds,
\end{align*}
for $\e$ sufficiently small.
And similarly,$$\int_\kk^\kl\lvert\rcoe\rvert^2ds\geq(1-\alpha)\int_\kk^\kl B^2(\varphi_s)ds.$$

\textbf{\romannumeral 3.} We proceed to estimate the conditional probability that $H(X_t^\e)$ follows $\varphi_t$ on the interval $[\kk,\kl]$, given that $H(X_\kk^\e)$ is close enough to $\varphi_\kk$. 
Let $c_k=\int_\kk^\kl B^2(\varphi_t)dt$, $s_0=(1-\alpha)c_k\e^\beta$, $s_1=(1+\alpha)c_k\e^\beta$, $P_k=\varphi_{\kl}-H(X_\kk^\e)$, and 
$$Z_t=\e^{\beta/2}\int_\kk^{\kk+t}\rcoe dW_s=\Tilde W(\langle Z\rangle_t),$$ 
where $\langle Z\rangle_t=\e^\beta\int_\kk^{\kk+t}\lvert\rcoe\rvert^2ds$ and $\Tilde W$ is a one-dimensional Wiener process. Earlier, we saw that $s_0\leq\langle Z\rangle_{\frac{1}{n}}\leq s_1$ on the event $A_k\cap\left\{\lvert\varphi_\kk-H(X_\kk^\e)\rvert<h\right\}$.
On the other hand, the increment in $H(X_t^\e)$ is almost $Z_{1/n}$, since for $\e$ small enough
$$\left|H(X_\kl^\e)-H(X_{\kk}^\e)-Z_{1/n}\right|=\e^\beta\left|\int_\kk^{\kl}AH(X_s^\e)ds\right|\leq\e^\beta\frac{M}{n}<h'/2.$$
Then, for the initial point $X_k^\e=x$ satisfying $|H(x)-\varphi_\kk|<h$,
\begin{align*}
    &{\bm{\mathrm P}}\left(\left\{\lvert\varphi_\kl-H(X_\kl^\e)\rvert<h'\right\}\cap A_k\right)\\
    =\ &{\bm{\mathrm P}}\left(\left\{\lvert H(X_\kl^\e)-H(X_\kk^\e)-P_k\rvert<h'\right\}\cap A_k\right)\\
    \geq\ &{\bm{\mathrm P}}\left(\left\{\left|Z_{1/n}-P_k\right|<\frac{h'}{2}\right\}\cap A_k\right)\\
    =\ &{\bm{\mathrm P}}\left(\left\{\left|\Tilde W(\langle Z\rangle_{\frac{1}{n}})-P_k\right|<\frac{h'}{2}\right\}\cap\left\{s_0\leq\langle Z\rangle_{\frac{1}{n}}\leq s_1\right\}\cap A_k\right)\\
    \geq\ &{\bm{\mathrm P}}\left(\left\{\sup_{s_0\leq s\leq s_1}\lvert \Tilde W_s-P_k\rvert<\frac{h'}{2}\right\}\cap A_k\right)\\
    \geq\ &{\bm{\mathrm P}}\left(\sup_{s_0\leq s\leq s_1}\lvert \Tilde W_s-P_k\rvert<\frac{h'}{2}\right)-\lambda_3(\e)\\
    \geq\ &{\bm{\mathrm P}}\left(\lvert \Tilde W_{s_0}-P_k\rvert<\frac{h'}{4},\sup_{s_0\leq s\leq s_1}\lvert \Tilde W_{s}-\Tilde W_{s_0}\rvert<\frac{h'}{4}\right)-\lambda_3(\e)\\
    =\ &\frac{1}{\sqrt{2\pi s_0}}\int_{P_k-h'/4}^{P_k+h'/4}\exp(-\frac{y^2}{2s_0})dy\cdot {\bm{\mathrm P}}\left(\sup_{s_0\leq s\leq s_1}\lvert \Tilde W_{s}-\Tilde W_{s_0}\rvert<\frac{h'}{4}\right)-\lambda_3(\e)\\
    \geq\ &\frac{h'}{2}\frac{1}{\sqrt{2\pi s_0}}\exp\left(-\frac{1}{2s_0}\left(\lvert\varphi_\kl-\varphi_\kk\rvert+2h\right)^2\right)\cdot\frac{1}{2}-\lambda_3(\e)\\
    \geq\ &2\exp\left(-\frac{\left(\lvert\varphi_\kl-\varphi_\kk\rvert+2h\right)^2}{2(1-\alpha)\e^\beta\int_\kk^\kl B^2(\varphi_t)dt}\right)-\exp\left(-\frac{n{\delta'}^2}{64M\e^\beta}\right)\\
    \geq\ &\exp\left(-\frac{\left(\lvert\varphi_\kl-\varphi_\kk\rvert+2h\right)^2}{2(1-\alpha)\e^\beta\int_\kk^\kl B^2(\varphi_t)dt}\right).
\end{align*}

\textbf{\romannumeral 4.} Since the previous result is valid for every $0\leq k<n$,
\begin{align*}
    &{\bm{\mathrm P}}\left(|H(X_1^\e)-\varphi_1|<h',\ \rho_{0,1}(H(X_t^\e),\varphi_t)<\delta\right)\\
    \geq\ &{\bm{\mathrm P}}\left(|H(X_1^\e)-\varphi_1|<h',\ \sup_{0\leq t\leq 1}\lvert H(X_t^\e)-\varphi_t\rvert<\delta'\right)\\
    \geq\ &{\bm{\mathrm P}}\left(\bigcap_{k=0}^{n-1}A_k\cap\left\{\lvert H(X_\kk^\e)-\varphi_\kk\rvert<h'\text{  for every   }1\leq k\leq n \right\}\right)\\
    \geq\ &\prod_{k=0}^{n-1}\inf_{y:|H(y)-\varphi_\kk|<h}{\bm{\mathrm P}}\left(\left\{\lvert\varphi_\kl-H(X_\kl^\e)\rvert<h'\right\}\cap A_k\mathrel{\stretchto{\mid}{4ex}} X_\kk^\e=y\right)\\
    \geq\ &\prod_{k=0}^{n-1}\exp\left(-\frac{\left(\lvert\varphi_\kl-\varphi_\kk\rvert+2h\right)^2}{2(1-\alpha)\e^\beta\int_\kk^\kl B^2(\varphi_t)dt}\right)\\
    =\ &\exp\left(-\frac{1}{2\e^\beta(1-\alpha)}\sum_{k=0}^{n-1}\left(\frac{\left|\int_\kk^\kl \dot\varphi_tdt\right|^2}{\int_\kk^\kl B^2(\varphi_t)dt}+4h\frac{h+\left|\int_\kk^\kl \dot\varphi_tdt\right|}{\int_\kk^\kl B^2(\varphi_t)dt}\right)\right)\\
    \geq\ &\exp\left(-\frac{1}{2\e^\beta(1-\alpha)}\sum_{k=0}^{n-1}\left(\int_\kk^\kl\frac{\lvert\dot\varphi_t\rvert^2}{B^2(\varphi_t)}dt+\frac{nh\delta'}{m}\right)\right)\\
    \geq\ &\exp\left(-\frac{1}{2\e^\beta}\left(\frac{1}{1-\alpha}\int_0^1\frac{\lvert\dot\varphi_t\rvert^2}{B^2(\varphi_t)}dt+\frac{1}{1-\alpha}\cdot\frac{n^2h\delta'}{m}\right)\right)\\
    \geq\ &\exp\left(-\frac{1}{2\e^\beta}\left(\int_0^1\frac{\lvert\dot\varphi_t\rvert^2}{B^2(\varphi_t)}dt+\frac{\alpha}{1-\alpha}\int_0^1\frac{\lvert\dot\varphi_t\rvert^2}{B^2(\varphi_t)}dt+\frac{1}{1-\alpha}\cdot\frac{n^2h\delta'}{m}\right)\right)\\
    \geq\ &\exp\left(-\e^{-\beta}\left(S(\varphi)+\gamma\right)\right).
\end{align*}

\end{proof}

\subsection{The case where the process starts at an exterior vertex}
In this section, we assume that $\bm{\mathrm  x}_0=(0,0)$ (the origin in $\mathbb R^2$) is a local minimum point and $H(\bm{\mathrm  x}_0)=0$, without loss of generality. We aim to estimate from below the probability that the random process starting at an exterior vertex escapes from a certain neighborhood of the vertex sufficiently fast. We start with the following lemma. Again, we simplify the notations as what we did in the preceding section. 

\begin{lemma}
Suppose that $X^\e_0=\bm{\mathrm x}_0$. Then, for every $T>0$, there exists a positive number $k$ such that the stopping time $\tau_1:=\inf\{t:H(X_t^\e)=k\e^\beta\}$ satisfies ${\bm{\mathrm P}}(\tau_1<T)\geq\frac{1}{2}$ for all $\e$ small enough.
\label{firststep1}
\end{lemma}
\begin{proof}
 Let us again write down the equation for $H(X_t^\e)$,$$H(X_t^\e)=\e^\beta\int_0^t AH(X_s^\e)ds+\e^{\beta/2}\int_0^t\rcoe dW_s,$$where $A$ is the operator $Au(x)=\frac{1}{2}\sum_{i,j}[\sigma(x)\sigma^*(x)]_{i,j}\cdot\frac{\partial^2}{\partial x_i\partial x_j}u(x)$. 
 First, note that both $\sigma\sigma^*$ and the Hessian matrix of $H$ are positive-definite at $\bm{\mathrm x}_0$, hence $d:=AH(\bm{\mathrm x}_0)/2>0$ and $AH(x)>d$ in a sufficiently small neighborhood of $\bm{\mathrm x}_0$. 
 Second, let $\lambda>0$ be the smaller eigenvalue of the Hessian matrix of $H$ at $\bm{\mathrm x}_0$.
 Then, by Taylor's expansion at $\bm{\mathrm x}_0$, we deduce that $|H(x)|\geq\frac{1}{4}\lambda|x|^2$ in a sufficiently small neighborhood of $\bm{\mathrm x}_0$. 
 Third, due to the boundedness of both the second derivatives of $H$ and the norm of $\sigma\sigma^*$, we have $|\nabla H^*(x)\sigma(x)|^2\leq\Lambda|x|^2$ for some $\Lambda$ and all $x\in\mathbb R^2$.
 Let $k$ be small enough so that $k\leq d/2$,
 $$\sqrt{\frac{32k\Lambda T}{\pi\lambda d^2}}\exp(-\frac{\lambda d^2}{32k\Lambda T})<\frac{1}{2},$$
 and, in $U=\{x:|x|\leq\sqrt\frac{4k}{\lambda}\e^{\beta/2}\}$, we have $AH(x)>d$ and $|H(x)|\geq\frac{1}{4}\lambda|x|^2$. We define $\tau_2=\inf\{t:H(X_t^\e)\in\partial U\}$, and it follows that $\tau_1\leq\tau_2$. We are ready to estimate
 \begin{align*}
     {\bm{\mathrm P}}(\tau_1\geq T)&= {\bm{\mathrm P}}\left(\tau_2\geq\tau_1\geq T,\ \e^\beta\int_0^T AH(X_s^\e)ds+\e^{\beta/2}\int_0^T\rcoe dW_s<k\e^\beta\right)\\
     &\leq {\bm{\mathrm P}}\left(\tau_2\geq\tau_1\geq T,\ \e^{\beta/2}\int_0^T\rcoe dW_s<(k-d)\e^\beta\right)\\
     &\leq {\bm{\mathrm P}}\left(\tau_2\geq\tau_1\geq T,\ \int_0^T\rcoe dW_s<-\frac{d\e^{\beta/2}}{2}\right)\\
     &={\bm{\mathrm P}}\left(\tau_2\geq\tau_1\geq T,\ \Tilde W(\int_0^T|\rcoe|^2ds)<-\frac{d\e^{\beta/2}}{2}\right)\\
     &\leq {\bm{\mathrm P}}\left(\tau_2\geq\tau_1\geq T,\ \inf_{0\leq t\leq 4k\Lambda T\e^\beta/\lambda} \Tilde W_t<-\frac{d\e^{\beta/2}}{2}\right)\\
     &\leq 2{\bm{\mathrm P}}\left(\Tilde W_{4k\Lambda T\e^\beta/\lambda}>\frac{d\e^{\beta/2}}{2}\right)\\
     &\leq \sqrt{\frac{32k\Lambda T}{\pi\lambda d^2}}\exp(-\frac{\lambda d^2}{32k\Lambda T})\\
     &<\frac{1}{2}.
 \end{align*}
\end{proof}
Assuming that the process starts at $x$ with $H(x)=k\e^\beta$, i.e. a certain distance away from the extremum point, we can apply Ito's formula with $f(x)=\sqrt x$ to the process $H(X_t^\e)$ in order to make the diffusion coefficient uniformly positive. 

\begin{lemma}
For a given non-constant $\varphi$ with $\varphi_0=H(x_0)=0$ and $S(\varphi)<\infty$ and a positive constant $c$, there exists $\rho>0$ such that the stopping time $\tau:=\inf\{t:H(X_t^\e)=\rho\}$ satisfies ${\bm{\mathrm P}}(\tau<T)\geq\exp(-c\e^{-\beta})$ for all $\e$ sufficiently small, where $T=\inf\{t:|\varphi_t-\varphi_0|=\rho\}$.
\label{extreme}
\end{lemma}
\begin{proof}
 
 Let $\tau_1=\inf\{t:H(X_t^\e)=k\e^\beta\}$ and let $\tau_2=\tau-\tau_1$. By the strong Markov property of $X_t^\e$, we deduce that 
 $${\bm{\mathrm P}}(\tau<T)\geq {\bm{\mathrm P}}(\tau_1<\frac{T}{2},\ \tau-\tau_1<\frac{T}{2}).$$
 By Lemma \ref{firststep1}, we can find such $k$ that $${\bm{\mathrm P}}(\tau_1<\frac{T}{2})\geq\frac{1}{2}.$$
 Consider the process that starts at $x$ satisfying $H(x)=k\e^\beta$. Since a two-dimensional nondegenerate diffusion does not reach the origin, we can apply Ito's formula to $\sqrt{H(X_t^\e)}$ and get
\begin{align*}
    \sqrt{H( X_t^\e)}=\sqrt{k\e^\beta}&+\e^\beta\int_0^t\left(\frac{AH( X_s^\e)}{2\sqrt{H( X_s^\e)}}-\frac{|\rcoe|^2}{8\sqrt{H( X_s^\e)^3}}\right)ds\\ &+\e^{\beta/2}\int_0^t\frac{\rcoe}{2\sqrt{H( X_s^\e)}}dW_s.
\end{align*}  
Due to the positive-definiteness of the Hessian matrix at $\bm{\mathrm x}_0=(0,0)$, we can find $r>0$ small enough that satisfies the conditions in Lemma \ref{positivedrift}, and there exist $\lambda_1,\lambda_2,\Lambda_1,\Lambda_2>0$ independent of $r$ such that, for every $x\in B(\bm{\mathrm x}_0,r)$, we have the estimates: $\lambda_1 |x|^2\leq H(x)\leq\lambda_2 |x|^2$ and $\Lambda_1|x|\leq|\nabla H(x)|\leq\Lambda_2 |x|$. Let $\rho$ be sufficiently small so that $X_t^\e\in B(\bm{\mathrm x}_0,r)$ for $t<\tau_2$. By Lemma \ref{positivedrift},
 \begin{align*}
     &\ {\bm{\mathrm P}}\left(\tau_2\geq T/2\right)\\ =&\ {\bm{\mathrm P}}\left(\sqrt{H( X^\e_{T/2})}\leq\sqrt\rho,\ \tau_2\geq T/2\right)\\
     \leq&\ {\bm{\mathrm P}}\left(\e^{\beta/2}\int_0^{T/2}\frac{ \rcoe}{2\sqrt{H( X_s^\e)}}dW_s\leq\sqrt\rho,\ \tau_2\geq T/2 \right)\\
     \leq&\ {\bm{\mathrm P}}\left(\Tilde W\left(\int_0^{T/2}\left\lvert\frac{ \rcoe}{2\sqrt{H( X_s^\e)}}\right\rvert^2ds\right)\leq\sqrt\rho\e^{-\beta/2},\ \tau_2\geq T/2 \right)\\
     \leq &\ {\bm{\mathrm P}}\left(\Tilde W_t\leq\sqrt\rho\e^{-\beta/2},\text{ for some }\lambda T<t<2\lambda T,\ \tau_2\geq T/2 \right)\\
     \leq&\ {\bm{\mathrm P}}\left(\Tilde W_t\leq\sqrt\rho\e^{-\beta/2},\text{ for some }\lambda T<t<2\lambda T\right).
 \end{align*}
 Here $\lambda$ is a constant determined by $\lambda_1,\lambda_2,\Lambda_1,\Lambda_2$. Hence 
  \begin{align*}
     {\bm{\mathrm P}}\left(\tau_2\leq T/2\right)\geq&\ {\bm{\mathrm P}}\left(W_t>\sqrt\rho\e^{-\beta/2},\text{ for all }\lambda T<t<2\lambda T\right)\\
     \geq&\ \frac{1}{2}{\bm{\mathrm P}}\left(W_{\lambda T}>\sqrt{2\rho}\e^{-\beta/2}\right)\\
     \geq&\ \frac{1}{\sqrt{2\pi\lambda T}}\exp(-\frac{3\rho}{\lambda T\e^\beta}).
 \end{align*}
 By Lemma \ref{zeroderivative}, we can find $\rho$ such that the value above is greater than $2\exp(-c\e^{-\beta})$.
\end{proof}

\input{separatrix}

\subsection{Local Large Deviation Principle: Lower Bound}
With all the preliminary results obtained, we can get the lower bound for the large deviation principle. 
\begin{lemma}
For any $\delta>0$, $\gamma>0$, and $\bm{\varphi}\in\textbf{C}\left([0,T],\Gamma\right)$ with $\bm{\varphi_0}=Y(x_0)$, there exists $\e_0$ such that
 $${\bm{\mathrm P}}\left(\rho_{0,T}(Y(X_t^\e),\bm{\varphi})<\delta\right)\geq\exp\left[-\e^{-\beta}(S(\bm{\varphi})+\gamma)\right],$$for any $\e<\e_0$.
\label{case1}
\end{lemma}
\begin{proof}
 We can assume that $\varphi$ is absolutely continuous and $\varphi'\in L^2([0,T])$ for the rest of the proof, because $S(\bm{\varphi})=\infty$ for any other $\bm{\varphi}=(i,\varphi)$. 
 Consider the case where $\bm{\varphi}_0$ is at an exterior vertex $O_k$ and $\bm{\varphi}_T$ is not at any vertex. 
 Given $\delta>0$ and $\gamma>0$, by Lemma~\ref{extreme}, we can find $\rho<\delta/4\wedge \inf_{k}|\varphi_T-H(\bm{\mathrm x}_k)|$ such that, for $\e$ sufficiently small,
 \begin{equation}
 \label{case1eq:1}
     {\bm{\mathrm P}}(\tau<q)\geq\exp(-\e^{-\beta}\gamma/4),
 \end{equation} where $q=\inf\{t:r(\bm{\varphi}_t,O_k)=\rho\}$ and $\tau=\inf\{t:|H(X_t^\e)-H(\bm{\mathrm x}_k)|=\rho\}$. Let $\delta'=\frac{1}{2}\rho$, and define the following three sequences inductively:
\begin{align*}
    q_0&=q;\\
    t_k&=\inf\left\{t>q_{k-1}:r(\bm{\varphi}_t,O)=\delta'\text{ for some vertex }O\right\},\\
    p_k&=\sup\left\{t<t_k:r(\bm{\varphi}_t,O)=2\delta'\text{ for some vertex }O\right\},\\
    q_k&=\inf\left\{t>t_k:r(\bm{\varphi}_t,O)=2\delta'\text{ for some vertex }O\right\},
\end{align*}
\begin{figure}[!ht] 
    \centering
    \includegraphics[width=0.6\textwidth]{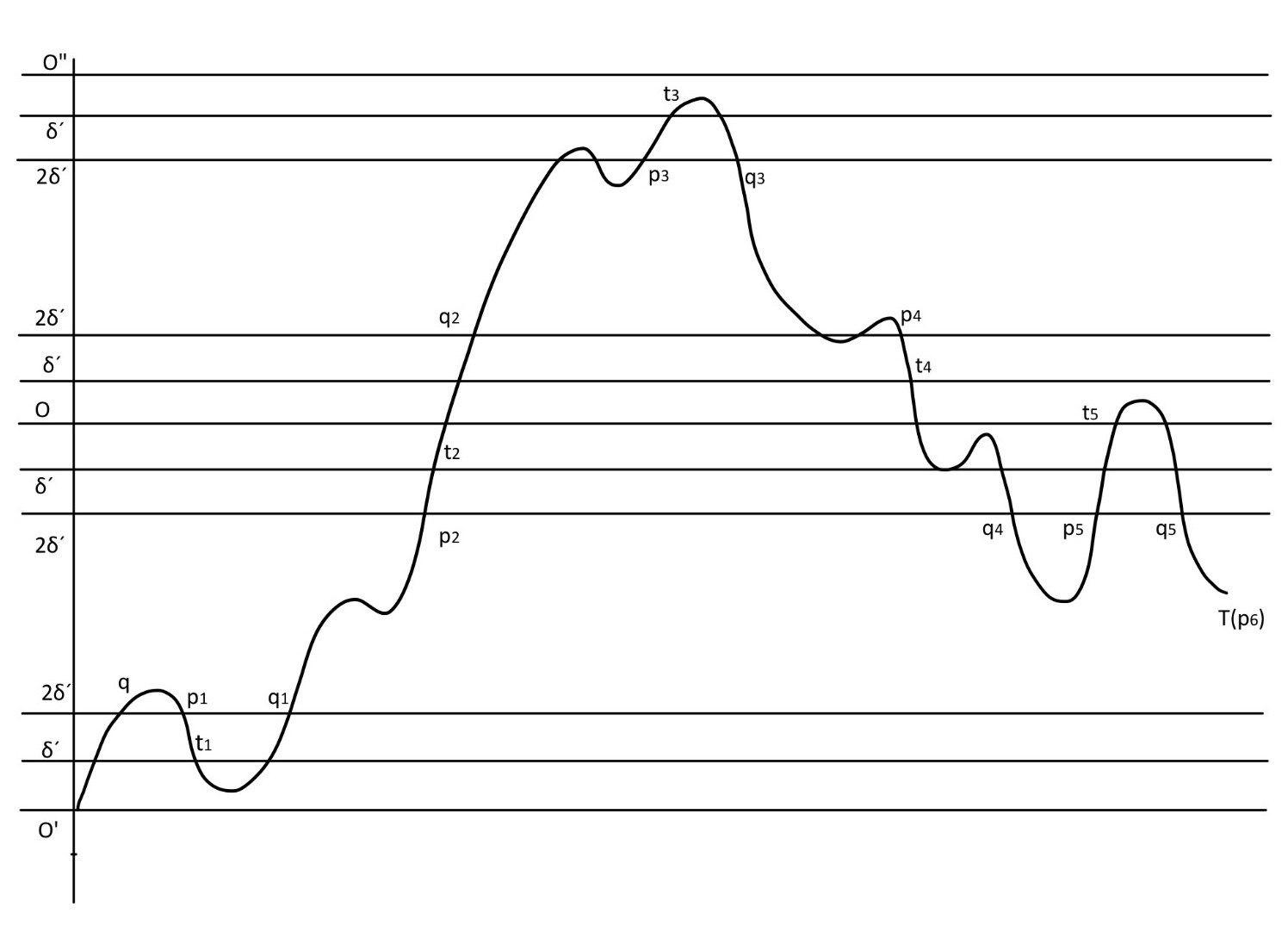}
    \caption{A typical partition of a trajectory.}
    \label{partition}
\end{figure}for $k\leq n$, where $n$ is such that the set $\left\{t>q_n:r(\bm{\varphi}_t,O)=\delta'\text{ for some vertex }O\right\}$ is empty. Note that $n\geq 0$ is finite due to the absolutely continuity of $\varphi$. We define $p_{n+1}=T$.
By Lemma \ref{loclem}, we can find $h<\frac{1}{2}\delta'$ such that, for all $h'<h$, there exists $\e_0$ small enough such that, conditioned on $X_{q_{k-1}}^\e=x$ satisfying $|H(x)-\varphi_{q_{k-1}}|<h$ and for all $\e<\e_0$,
\begin{align*}
&{\bm{\mathrm P}}\left(|H(X_{p_{k}}^\e)-\varphi_{p_{k}}|<h',\ \rho_{q_{k-1},p_{k}}(Y(X_t^\e),\bm{\varphi}_t)<\delta\right)
\geq\exp\left[-\e^{-\beta}(S_{q_{k-1},q_{k}}(\bm{\varphi})+\frac{\gamma}{2^{k+2}})\right]
\end{align*}
holds for all $1\leq k\leq n+1$. By the strong Markov property of $X_t^\e$ (since the process can remain near the same level curve between time $\tau$ and $q$), for each $h'>0$, there exists $\e_0$ such that, for all $\e<\e_0$,
$${\bm{\mathrm P}}(|H(X_{q}^\e)-\varphi_{q}|<h',\ \rho_{0,q}(Y(X_t^\e),\bm{\varphi}_t)<\delta)\geq\exp(-\e^{-\beta}\gamma/2).$$
On the other hand, for $1\leq k\leq n$, on the interval $[p_k,q_k]$, we have two different situations:\\ 
(\romannumeral1) If $\bm{\varphi}_{p_k}=\bm{\varphi}_{q_k}$, then, conditioned on $X_{p_{k}}^\e=x$ satisfying $|H(x)-\varphi_{p_{k}}|<h/2$, for all $\e$ small enough,
\begin{align*}
    &{\bm{\mathrm P}}\left(|H(X_{q_{k}}^\e)-\varphi_{q_{k}}|<h,\ \rho_{p_{k},q_{k}}(Y(X_t^\e),\bm{\varphi}_t)<\delta\right)\\
    \geq\ &{\bm{\mathrm P}}\left(\sup_{p_k\leq t\leq q_k}|H(X_t^\e)-H(X_{p_k}^\e)|<h/2\right)>\frac{1}{2}.
\end{align*}
(\romannumeral2) If $\bm{\varphi}_{p_k}\not=\bm{\varphi}_{q_k}$, then, for $\e$ small enough, by Lemma~\ref{separatrix3}, there exists $\mu<h/2$ such that, conditioned on $X_{p_{k}}^\e=x$ satisfying $|H(x)-\varphi_{p_{k}}|<\mu$,
\begin{align*}
&{\bm{\mathrm P}}\left(|H(X_{q_{k}}^\e)-\varphi_{q_{k}}|<h,\ \rho_{p_{k},q_{k}}(Y(X_t^\e),\bm{\varphi}_t)<\delta\right)\geq\exp(-\e^{-\beta}(S_{p_kq_k}(\bm{\varphi})+\frac{\gamma}{2^{k+2}})).
\end{align*}
With the estimates above combined, we obtain that
$${\bm{\mathrm P}}\left(\rho_{0,T}(Y(X_t^\e),\bm{\varphi})<\delta\right)\geq\exp\left[-\e^{-\beta}(S(\bm{\varphi})+\gamma)\right]$$
for any $\e$ small enough. In the case where $\bm{\varphi}_0$ is at an interior vertex or inside an edge, or $\bm{\varphi}_T$ is at a vertex, the proof is similar.
\end{proof}

\begin{theorem}
\label{lowermaintheorem}
Given $\bm{\varphi}\in\textbf{C}\left([0,T],\Gamma\right)$ with $\bm{\varphi_0}=Y(x_0)$,
$$\varliminf_{\delta\to0}\varliminf_{\e\to 0}\e^\beta\log {\bm{\mathrm P}}\left(\rho_{0,T}(Y(X_t^\e),\bm{\varphi}))<\delta\right)\geq-S(\bm{\varphi}).$$
\end{theorem}
\begin{proof}
This is an immediate consequence of the previous lemma.
\end{proof}

%% file: separatrix.tex
\subsection{Estimates near the Separatrix}
In this section, we prove that the probability that the process goes through an interior vertex sufficiently fast is not too small. For that purpose, we invoke the Morse-Palais Lemma to make concrete computations. Suppose that $\bm{\mathrm x}_k$ is a nondegenerate saddle point of $H$ and that $H(\bm{\mathrm x}_k)=0$. By the Morse-Palais Lemma, there exists a neighborhood $U_0$ (we may assume that $U_0$ is an $8l\times4l$ rectangle with $l<1$) of the origin and a diffeomorphism $\psi:U_0\to V_0$, where $V_0$ is a neighbourhood of $\bm{\mathrm x}_k$, such that $H(\psi(\mu,\nu))=\mu^2-\nu^2=:G(\mu,\nu)$, if $(\mu,\nu)\in U_0$. Since $\psi$ is a diffeomorphism, there exists $\overline M$ such that $\|\psi\|$, $\|\psi^{-1}\|$, $\|J_\psi\|$, $\|J_{\psi^{-1}}\|$, $|\mathrm {det}(J_\psi)|$, $|\nabla \mathrm {det}(J_\psi)|$ are all bounded by $\overline M$, for all $(\mu,\nu)\in U_0$. 
Let us take $U\subset U_0$ to be the $4l\times 2l$ rectangle centered at the origin, and $V=\psi(U)$. Look at the deterministic slow motion
$$dx_t=\nabla^\perp H(x_t)dt.$$
Then we derive the equation for the process $\psi^{-1}(x_t)$ in $U_0$:
$$\frac{d\psi^{-1}(x_t)}{dt}=\frac{1}{\mathrm {det}(J_\psi(\psi^{-1}(x_t)))}\cdot\nabla^\perp G(\psi^{-1}(x_t)).$$
This means that $G(\psi^{-1}(x_t))$ remains constant as the process evolves, which leads to the fact that if $x_0$ is close to $\bm{\mathrm x}_k$ and $H(x_0)>0$, then the deterministic motion $\psi^{-1}(x_t) $ will exit from $U$ across the upper or lower boundary, depending on which side of separatrix $x_0$ is on and the sign of $\mathrm{det} (J_\psi)$. We define, for $(\mu,\nu)\in U_0$, 
\begin{equation}
\label{eq:Tmunu}
    T(\mu,\nu)=\frac{1}{2}\int_{\nu}^l\frac{\mathrm {det}(J_\psi(\sqrt{y^2+\mu^2-\nu^2},y))}{\sqrt{y^2+\mu^2-\nu^2}}dy.
\end{equation}
Without loss of generality, we assume $\mathrm{det} (J_\psi)>0$.
Then it's not hard to see that $T(\mu,\nu)=\inf\{t:\psi^{-1}(x_t)\not\in U,\ \psi^{-1}(x_0)=(\mu,\nu)\}$ if $(\mu,\nu)\in U$, $\mu>0$, and $0<H(\psi((\mu,\nu)))<3l^2$
(in this case, the process $\psi^{-1}(x_t)$ exits from $U$ across the upper part of the boundary, and the integrand in \eqref{eq:Tmunu} (multiplied by a half) is the reciprocal of the vertical speed); and that $T(\psi^{-1}(x_t))-T(\psi^{-1}(x_{t+1}))=1$ if $x_t,\ x_{t+1}\in V_0$ and $0<H(x_t)<3l^2$; and for $x_0\in V$ and $0<H(x_0)<3l^2$, 
\begin{equation}
\label{eq:T}
    T(\psi^{-1}(x_0))\leq\overline{M}\left[\log\left(l+\sqrt{l^2+H(x_0)}\right)-\frac{1}{2}\log(H(x_0))\right].
\end{equation}
Furthermore, differentiating $T(\mu,\nu)$, we get the following estimates:
\begin{align}
    |\frac{\partial T}{\partial \mu}|&\leq \overline M\left[\frac{|\pi\mu|}{\sqrt{\mu^2-\nu^2}}+\frac{|\mu|}{(\mu^2-\nu^2)\sqrt{1+\mu^2-\nu^2}}+\frac{|\nu|}{\mu^2-\nu^2}\right]\leq\frac{C}{H(\psi(\mu,\nu))},\label{derv1}\\
    |\frac{\partial T}{\partial \nu}|&\leq \overline M\left[\frac{|\pi\nu|}{\sqrt{\mu^2-\nu^2}}+\frac{|\nu|}{(\mu^2-\nu^2)\sqrt{1+\mu^2-\nu^2}}+\frac{|\mu|}{\mu^2-\nu^2}\right]\leq\frac{C}{H(\psi(\mu,\nu))},\label{derv2}
\end{align}for a certain positive constant $C$ (details are given in the Appendix, Part 4).
We will go through the following steps to get the desired estimates on the probability of going through an interior vertex sufficiently fast:
\begin{enumerate}
\item We first deal with the random process that starts on the separatrix. Then, by Theorem 3.2 in \cite{Hairer2016AFK}, for any $d>0$, the process escapes to $C_{ki}(\e^{\frac{1+d}{2}\beta})$ in a fixed time with a probability bounded from below independently of $\e$.
\item If $d>0$ is small enough, then the process that starts at $C_{ki}(\e^{\frac{1+d}{2}\beta})$ can reach $C_{ki}(\e^{\frac{1-d}{2}\beta})$ with at least an exponentially small probability, as is stated in Lemma~\ref{firststep}.
\item After the process arrives at $C_{ki}(\e^{a\beta})$ with $a<1/2$, it can further travel to $C_{ki}(\e^{(a-d)\beta})$ with $d$ small but positive, with at least an exponentially small probability. This is the statement of Lemma~\ref{steps}. Hence, after finitely many steps, the process can reach $C_{ki}(\e^{d\beta})$.
\item The last step would be traveling from $C_{ki}(\e^{d\beta})$ to a constant distance from the separatrix. This is done in Lemma~\ref{laststep}.
\item Similarly, the process that starts at a certain distance from the separatrix can hit it within a constant time with at least an exponentially small probability. Therefore, the probability that the process follows a certain trajectory going across the separatrix is not terribly small, which is the result of Lemma~\ref{separatrix3}.
\end{enumerate}
Let us define the curves:
\begin{align*}
    \gamma_{in}&=\{x\in\partial V\cap D_i:  \nabla^\perp H(x)\text{ points inwards of }V\},\\
    \gamma_{out}&=\{x\in\partial V\cap D_i:  \nabla^\perp H(x)\text{ points outwards of }V\}.
\end{align*}
The following lemma contains simple estimates on the Brownian motion that will be needed later. The proof will be given in the Appendix, Part 4.
\begin{lemma}
\label{brownianatsaddle}
(\romannumeral1) For $0<d<a<\frac{1}{2}$ and every $\kappa>0$, we have the following estimate for $\e$ small enough:
\begin{align*}
    &\ {\bm{\mathrm P}}\left(\e^{\beta/2}{W}_{T}<-4\e^{(a-d)\beta}(a-d)\beta|\log\e|,\ \e^{\beta/2} W_t<\frac{1}{4}a\beta|\log\e|\e^{a\beta},\ \forall\ 0<t<T\right)\\
    \geq &\ 2\exp\left(-\e^{-(1-2(a-d))\beta-\kappa}\right).
\end{align*}
(\romannumeral2) For $0<d<\frac{1/\beta-1}{3}\wedge\frac{1}{2}$ and every $\kappa>0$, we have the following estimate for $\e$ small enough:
$${\bm{\mathrm P}}\begin{pmatrix}
\e^{\beta/2}{W}_T<-2(1-d)\beta\e^{\frac{1-d}{2}\beta}|\log\e|;\\ \e^{\beta/2}{W}_t<\e^{\frac{1+d}{2}\beta},\ \forall\ 0<t<\e^{d\beta};\\ \e^{\beta/2}{W}_t<-\e^{\frac{1-d}{2}\beta},\ \forall\ \e^{d\beta}<t<T
\end{pmatrix}
\geq 2\exp(-\e^{-2d\beta-\kappa}).$$
(\romannumeral3) For $0<d<\frac{1/\beta-1}{2}$ and every $A>0$, we have the following estimate for $\e$ small enough:
$${\bm{\mathrm P}}\left(\e^{\beta/2}{W}_{T}<-A,\ \e^{\beta/2} W_t<\frac{1}{4}d\beta|\log\e|\e^{d\beta},\ \forall\ 0<t<T\right)\geq 2\exp\left(-\frac{A^2}{T}\e^{-\beta}\right).$$
\end{lemma}
As we've already seen from \eqref{eq:T}, the exit time for the deterministic process from the neighborhood of a saddle point is approximately $\lvert\log H\rvert$. Now we need a similar statement for $X_t^\e$.
\begin{lemma}
\label{escapetime}
Suppose that $O_k$ is an interior vertex and an endpoint of $I_i$. For any given $0<a<\frac{1}{2}$, $0<d<\frac{1/\beta-1}{2}\wedge a$, and for every $0<\kappa<\frac{1}{8}(1-(1+2d)\beta)$, every $T>0$, we have the following estimate for each $x_0\in\gamma_{out}\cap D_{i}(\frac{9}{10}\e^{a\beta},\frac{11}{10}\e^{a\beta})$ and $\e$ small enough
$${\bm{\mathrm P}}\left(\int_0^t|\log H(X_s^\e)\rcoe|^2ds>t,\ \forall\ t\leq\eta\wedge T\right)\geq 1-\exp\left(-\e^{-(1-2(a-d))\beta-\kappa}\right),$$where $\eta=\inf\{t:X_t^\e\in C_{ki}(\e^{(a-d)\beta})\cup C_{ki}(\e^{a\beta}/2)\}.$
\end{lemma}
The idea behind this lemma is that the process starts on $\gamma_{out}$ and moves away from the saddle point first, with the intervals of time away from and near the saddle point alternating repeatedly. The integrand is large when the process is away from the saddle point, which is sufficient to compensate for the relatively long time spent near the saddle point. A detailed proof will be given in the Appendix, Part 4. Now let us formulate the argument of Step 3 (above) in the following lemma.
\begin{lemma}
\label{steps}
Suppose that $O_k$ is an interior vertex and an endpoint of $I_i$. For any given $0<a<\frac{1}{2}$, $0<d<\frac{1/\beta-1}{2}\wedge a$, and for every $T>0$, every $c>0$, there exists $\e_0$ small enough such that, for each $x_0\in\gamma_{out}\cap D_{i}(\frac{9}{10}\e^{a\beta},\frac{11}{10}\e^{a\beta})$, we have $${\bm{\mathrm P}}(\tau\leq T)\geq\exp{(-c\e^{-\beta})},\ \forall\e<\e_0,$$where $\tau=\inf\{t:X_t^\e\in C_{ki}(\e^{(a-d)\beta})\}.$
\end{lemma}
\begin{proof}
Define $\tau'=\inf\{t:X_t^\e\in C_{ki}(\e^{a\beta}/2)\}$, $\eta=\tau\wedge\tau'$, and $f(x)=x\log x-x$. 
Since $f\in C^2([\e^{a\beta}/2,\e^{(a-d)\beta}])$, Ito's formula can be applied here:
\begin{align*}
    f(H(X_\eta^\e))=&\ f(H(x_0))+\e^{\beta}\int_0^\eta\left(\log H(X_s^\e)AH(X_s^\e)+\frac{|\rcoe|^2}{2H(X_s^\e)}\right)ds\\ 
&+\e^{\beta/2}\int_0^\eta\log H(X_s^\e)\rcoe dW_s\\
=&\ f(H(x_0))+\e^{\beta}\int_0^\eta\left(\log H(X_s^\e)AH(X_s^\e)+\frac{|\rcoe|^2}{2H(X_s^\e)}\right)ds\\ 
&+\e^{\beta/2}\Tilde{W}\left(\int_0^\eta|\log H(X_s^\e)\rcoe|^2ds\right),
\end{align*} where $\Tilde W$ is another Brownian motion. To simplify notations, define events $R$ and $S$ as follows:
\begin{align*}
    R&=\left\{\e^{\beta/2}\Tilde{W}_{T}<-4\e^{(a-d)\beta}(a-d)\beta|\log\e|,\ \e^{\beta/2}\Tilde W_t<\frac{1}{4}a\beta|\log\e|\e^{a\beta}\ \text{for every}\ 0<t<T\right\},\\
    S&=\left\{\int_0^t|\log H(X_s^\e)\rcoe|^2ds>t,\ \forall\ t\leq\eta\wedge T\right\}.
\end{align*}
By Lemma~\ref{brownianatsaddle}(\romannumeral1) and Lemma~\ref{escapetime}, we know that ${\bm{\mathrm P}}(R\cap S)\geq\exp\left(-\e^{-(1-2(a-d))\beta-\kappa}\right)$ for any positive $\kappa$ and all sufficiently small $\e$. 
Therefore, since $\kappa$ can be taken arbitrarily small, it remains to show that $R\cap S$ implies $\tau\leq T$. 
On this event, if we assume that $\eta>T$, then 
$$\int_0^T|\log H(X_s^\e)\rcoe|^2ds>T.$$ So, we may define $\eta'=\inf\{t:\int_0^t|\log H(X_s^\e)\rcoe|^2ds=T\}<T<\eta$, and still write the equation
\begin{align*}
    f(H(X_{\eta'}^\e))=&\ f(H(x_0))+\e^{\beta}\int_0^{\eta'}\left(\log H(X_s^\e)AH(X_s^\e)+\frac{|\rcoe|^2}{2H(X_s^\e)}\right)ds\\ 
&+\e^{\beta/2}\Tilde{W}\left(\int_0^{\eta'}|\log H(X_s^\e)\rcoe|^2ds\right)\\
=&\ f(H(x_0))+\e^{\beta}\int_0^{\eta'}\left(\log H(X_s^\e)AH(X_s^\e)+\frac{|\rcoe|^2}{2H(X_s^\e)}\right)ds+\e^{\beta/2}\Tilde W_T.
\end{align*}
Since the Brownian motion here dominates the whole right hand side due to the definition of $R$,  we have $f(H(X_{\eta'}^\e))<f(\e^{(a-d)\beta})$, which contradicts the fact that $\eta'<\eta$. Thus $\eta\leq T$. Then it follows that $\tau\leq T$ since $f(H(X_\eta^\e))\not=f(\e^{a\beta}/2)$ by the definition of the set $R$.
\end{proof}
We can complete the proofs corresponding to Steps 2 and 4 above in the same way.
\begin{lemma}
\label{firstescapetime}
Suppose that $O_k$ is an interior vertex and an endpoint of $I_i$. For every $d<\frac{1/\beta-1}{3}\wedge \frac{1}{2}$, and for every $T>0$, we have the following estimate for each $x_0\in\gamma_{out}\cap D_{i}(\frac{9}{10}\e^{\frac{1+d}{2}\beta},\frac{11}{10}\e^{\frac{1+d}{2}\beta})$ and $\e$ small enough
$${\bm{\mathrm P}}\left(\int_0^t|\log H(X_s^\e)\rcoe|^2ds>t,\ \forall\ t\leq\eta\wedge T\right)\geq 1-\exp\left(-\e^{-2d\beta-\kappa}\right),$$where $\eta=\inf\{t:X_t^\e\in C_{ki}(\e^{\frac{1-d}{2}\beta})\cup C_{ki}(\e^{\frac{1+d}{2}\beta}/2)\}$.
\end{lemma}
Since the proof is similar to that of Lemma~\ref{escapetime}, the details are omitted.

\begin{lemma}
\label{firststep}
Suppose that $O_k$ is an interior vertex and an endpoint of $I_i$. For every $d<\frac{1/\beta-1}{3}\wedge \frac{1}{2}$, and for every $T>0$, $c>0$, there exists $\e_0$ small enough, such that for each $x_0\in\gamma_{out}\cap D_{i}(\frac{9}{10}\e^{\frac{1+d}{2}\beta},\frac{11}{10}\e^{\frac{1+d}{2}\beta})$, we have $${\bm{\mathrm P}}(\tau\leq T)\geq\exp{(-c\e^{-\beta})},\ \forall\e<\e_0,$$where $\tau=\inf\{t:X_t^\e\in C_{ki}(\e^{\frac{1-d}{2}\beta})\}.$
\end{lemma}
\begin{proof}
The proof is similar to that of Lemma~\ref{steps}. Define $\tau'=\inf\{t:X_t^\e\in C_{ki}(\e^{\frac{1+d}{2}\beta}/2)\}$, $\eta=\tau\wedge\tau'$, and $f(x)=x\log x-x$. 
Since $f\in C^2([\frac{1}{2}\e^{\frac{1+d}{2}\beta},\e^{\frac{1-d}{2}\beta}])$, Ito's formula can be applied here:
\begin{align*}
    f(H(X_\eta^\e))=&\ f(H(x_0))+\e^{\beta}\int_0^\eta\left(\log H(X_s^\e)AH(X_s^\e)+\frac{|\rcoe|^2}{2H(X_s^\e)}\right)ds\\ 
&+\e^{\beta/2}\Tilde{W}\left(\int_0^\eta|\log H(X_s^\e)\rcoe|^2ds\right).
\end{align*}
Let 
\begin{align*}
    R=&\begin{Bmatrix}
\e^{\beta/2}\Tilde{W}_T<-2(1-d)\beta\e^{\frac{1-d}{2}\beta}|\log\e|;\ \e^{\beta/2}\Tilde{W}_t<\e^{\frac{1+d}{2}\beta},\ \forall\ 0<t<\e^{d\beta};\\ \e^{\beta/2}\Tilde{W}_t<-\e^{\frac{1-d}{2}\beta},\ \forall\ \e^{d\beta}<t<T
\end{Bmatrix},\\
    S=&\left\{\int_0^t|\log H(X_s^\e)\rcoe|^2ds>t,\ \forall\ t\leq\eta\wedge T\right\}.
\end{align*}
By Lemma~\ref{brownianatsaddle}(\romannumeral2) and Lemma~\ref{firstescapetime}, we know that ${\bm{\mathrm P}}(R\cap S)\geq\exp\left(-\e^{-2d\beta-\kappa}\right)$ for any positive $\kappa$, and $R\cap S$ implies $\tau\leq T$, so we're done.
\end{proof}
The proof of the following lemma is also similar to that of Lemma~\ref{escapetime}.
\begin{lemma}
\label{lastescapetime}
Suppose that $O_k$ is an interior vertex and an endpoint of $I_i$. For every $d<\frac{1/\beta-1}{2}\wedge 1$, and for every $T>0$, we have the following estimate for each $x_0\in\gamma_{out}\cap D_{i}(\frac{9}{10}\e^{d\beta},\frac{11}{10}\e^{d\beta})$ and $\e$ small enough
$${\bm{\mathrm P}}\left(\int_0^t|\log H(X_s^\e)\rcoe|^2ds>t,\ \forall\ t\leq\eta\wedge T\right)\geq 1-\exp\left(-\e^{-\beta-\kappa}\right),$$where $\eta=\inf\{t:X_t^\e\in C_{ki}(\delta)\cup C_{ki}(\e^{d\beta}/2)\}$.
\end{lemma}

\begin{lemma}
\label{laststep}
Suppose that $O_k$ is an interior vertex and an endpoint of $I_i$. For every $d<\frac{1/\beta-1}{2}\wedge 1$, and for every $T>0$, $c>0$, there exist $\delta>0$ and $\e_0>0$ such that, for each $x_0\in\gamma_{out}\cap D_{i}(\frac{9}{10}\e^{d\beta},\frac{11}{10}\e^{d\beta})$, we have
$${\bm{\mathrm P}}(\tau\leq T)\geq\exp{(-c\e^{-\beta})},\ \forall\e<\e_0,$$where $\tau=\inf\{t:X_t^\e\in C_{ki}(\delta)\}.$
\end{lemma}
\begin{proof}
The proof is, again, similar to that of Lemma~\ref{steps}. Define $\tau'=\inf\{t:X_t^\e\in C_{ki}(\e^{d\beta}/2)\}$, $\eta=\tau\wedge\tau'$ and $f(x)=x\log x-x$. 
Since $f\in C^2([\frac{1}{2}\e^{d\beta},\delta])$, Ito's formula can be applied here:
\begin{align*}
    f(H(X_\eta^\e))=&\ f(H(x_0))+\e^{\beta}\int_0^\eta\left(\log H(X_s^\e)AH(X_s^\e)+\frac{|\rcoe|^2}{2H(X_s^\e)}\right)ds\\ 
&+\e^{\beta/2}\Tilde{W}\left(\int_0^\eta|\log H(X_s^\e)\rcoe|^2ds\right).
\end{align*}
Let 
\begin{align*}
    R=&\left\{\e^{\beta/2}\Tilde{W}_{T}<2f(\delta),\ \e^{\beta/2}\Tilde W_t<\frac{1}{4}d\beta|\log\e|\e^{d\beta}\ \text{for every}\ 0<t<T\right\},\\
    S=&\left\{\int_0^t|\log H(X_s^\e)\rcoe|^2ds>t,\ \forall\ t\leq\eta\wedge T\right\}.
\end{align*}
By Lemma~\ref{brownianatsaddle}(\romannumeral3) (with $A=-2f(\delta)$) and Lemma~\ref{lastescapetime}, we know that ${\bm{\mathrm P}}(R\cap S)\geq\exp\left(-\frac{4f(\delta)^2}{T}\e^{-\beta}\right)$, and $R\cap S$ implies $\tau\leq T$, so we're done if we choose $\delta$ so small that $4f(\delta)^2<cT$.
\end{proof}
The following lemma ensures that the assumptions that the process starts on a piece of $\gamma_{out}$ in the lemmas above make sense, since the process reaches the piece of the curve with a uniformly positive probability.
\begin{lemma}
\label{gammaout}
For $A<1/2$ and every fixed $T>0$, we have the following estimate for $\e$ small enough:
$${\bm{\mathrm P}}(\tau<T)\geq1/2,$$where $\tau=\inf\{t:X_t^\e\in\gamma_{out}\cap D_i(\frac{9}{10}\e^A,\frac{11}{10}\e^A)\}$, for each $x_0\in C_{ki}(\e^A)$.
\end{lemma}
\begin{proof}
The proof is similar to that of Lemma~\ref{escapetime}. Let us look at the slow motion:
$$    d\Tilde X_t^\e=\nabla^{\perp}H(\Tilde X_t^\e )dt+\sqrt\e\sigma(\Tilde X_t^\e)d\Tilde{W_t},\ \ \ \ \ \ \Tilde X_0^\e=x_0\in\mathbb R^2,$$
and take $\alpha\in(A,1/2)$. Then the idea is that $\tau<T$ on the event
$$E:=\bigcap_{k=0}^{2A\overline M|\log\e|}\left\{\sup_{\Delta\in[0,1]}|\sqrt{\e}\int_k^{k+\Delta}\sigma(\Tilde X^\e_s)dW_s|<\e^\alpha\right\},$$
and ${\bm{\mathrm P}}(E)\geq 1-2A\overline M|\log\e|\exp(-\e^{\alpha-1/2})\geq1/2$ for $\e$ sufficiently small, since $\sigma$ is bounded and the stochastic integral can be represented as a time changed Brownian motion. See the proof of Lemma~\ref{escapetime} for details.
\end{proof}

\begin{lemma}
\label{separatrix}
Suppose that $O_k$ is an interior vertex. For any $T>0$, $c>0$, there exist $\delta>0$ and $\e_0>0$ such that, for every $\e<\e_0$ and each $x_0\in C_k$,
$${\bm{\mathrm P}}(\tau\leq T)\geq\exp{(-c\e^{-\beta})},$$where $\tau=\inf\{t:X_t^\e\in C_{ki}(\delta)\}.$
\end{lemma}
\begin{proof}
Set $d=\frac{1/\beta-1}{4}\wedge\frac{1}{4}$ and $n=[\frac{1}{d}]+8$. By Theorem 3.2 in \cite{Hairer2016AFK}, for every $d>0$, the process reaches $C_{ki}(\e^{\frac{1+d}{2}\beta})$ by time $T/n$ with probability at least $p>0$, where $p$ is independent of $\e$. Then we apply Lemma~\ref{steps}, Lemma~\ref{firststep}, Lemma~\ref{laststep}, and Lemma~\ref{gammaout} with our fixed $d$, $\frac{T}{2n}$, and $\frac{c}{2n}$. By the strong Markov property of $X_t^\e$, we get the desired result.
\end{proof}
As in Lemma~\ref{separatrix}, the process starting at $C_{kj}(\delta)$ arrives at the separatrix with probability at least $\exp(-c\e^{-\beta})$ prior to time $T$. Combining these two results, we get the following lemma.
\begin{lemma}
\label{separatrix1}
Suppose that $O_k$ is an interior vertex with $O_k\sim I_i$ and $O_k\sim I_j$. Then, for any $T>0$, $c>0$, there exists $\delta$ such that, for any $x_0\in C_{kj}(\delta)$, $${\bm{\mathrm P}}(\tau\leq T)\geq\exp{(-c\e^{-\beta})},$$where $\tau=\inf\{t:X_t^\e\in C_{ki}(\delta)\}.$
\end{lemma}
Next, we formulate the lower bound for the probability that the random process goes through an interior vertex, near a certain deterministic path.

\begin{lemma}
\label{separatrix3}
Suppose that $O_k$ is an interior vertex with $O_k\sim I_i$ and $O_k\sim I_j$.
For every $\delta>0$, $\gamma>0$, $h>0$, and a trajectory ${\bm{\varphi}}$ such that ${\bm{\varphi}}_0=(j,H({\bm{\mathrm x}}_k)\pm\delta)$,  ${\bm{\varphi}}_T=(i,H({\bm{\mathrm x}}_k)\pm\delta)$, and $|\varphi_t-H({\bm{\mathrm x}}_k)|\leq\delta$ for $t\in[0,T]$,
there exist $h_0>0$ and $\e_0>0$ such that, for any $x_0$ satisfying $r(Y(x_0),{\bm{\varphi}}_0)<h_0$, we have the estimate for all $\e<\e_0$
$${\bm{\mathrm P}}\left(r(Y(X_T^\e),{\bm{\varphi}}_T)<h,\rho_{0T}(Y(X_t^\e),{\bm{\varphi}}_t)<2\delta\right)\geq\exp(-\e^{-\beta}(S_{0T}({\bm{\varphi}})+\gamma)).$$
\end{lemma}

\begin{proof}
Without loss of generality, we assume that $H({\bm{\mathrm x}}_k)=0$ and $h<\delta$. Due to the absolute continuity of $\varphi$, we can fix $\hat t$ so small that $|a-b|<\hat t$ implies that $|\varphi_a-\varphi_b|<h/2$.  
Fix $\delta'<\delta/2$ positive and small, such that for any $x_0\in C_{kj}(\delta')$, 
\begin{equation}
\label{separatrix3eq:1}
    {\bm{\mathrm P}}(\tau_{\delta'}\leq \hat t;\ X_t^\e\in D_k(\delta/2),\ t\in[0,\tau_{\delta'}])\geq\exp{(-\gamma\e^{-\beta}/4)},
\end{equation}
where $\tau_{\delta'}=\inf\{t:X_t^\e\in C_{ki}(\delta')\}.$ 
This is possible since, for $\delta'$ small enough, the process starting on $C_{kj}(\delta')$ escapes $D_k(\delta/2)$, within a fixed time, with a probability that is  bounded from above by $\exp(-C\e^{-\beta})$ for certain constant $C$, while the constant $c$ in Lemma~\ref{separatrix1} corresponding to $\delta'$ can be chosen such that $c<C$. 

Let $t_1=\inf\{t>0:\varphi_t=\pm\delta'/2\}$, $t_2=\sup\{t>0:\varphi_t=\pm\delta'\}$, and $\Tilde\varphi_t=\varphi_{t-\hat t}$ for $t\in[t_2+\hat t, T]$.
\begin{figure}[!ht] 
    \centering
    \includegraphics[width=0.6\textwidth]{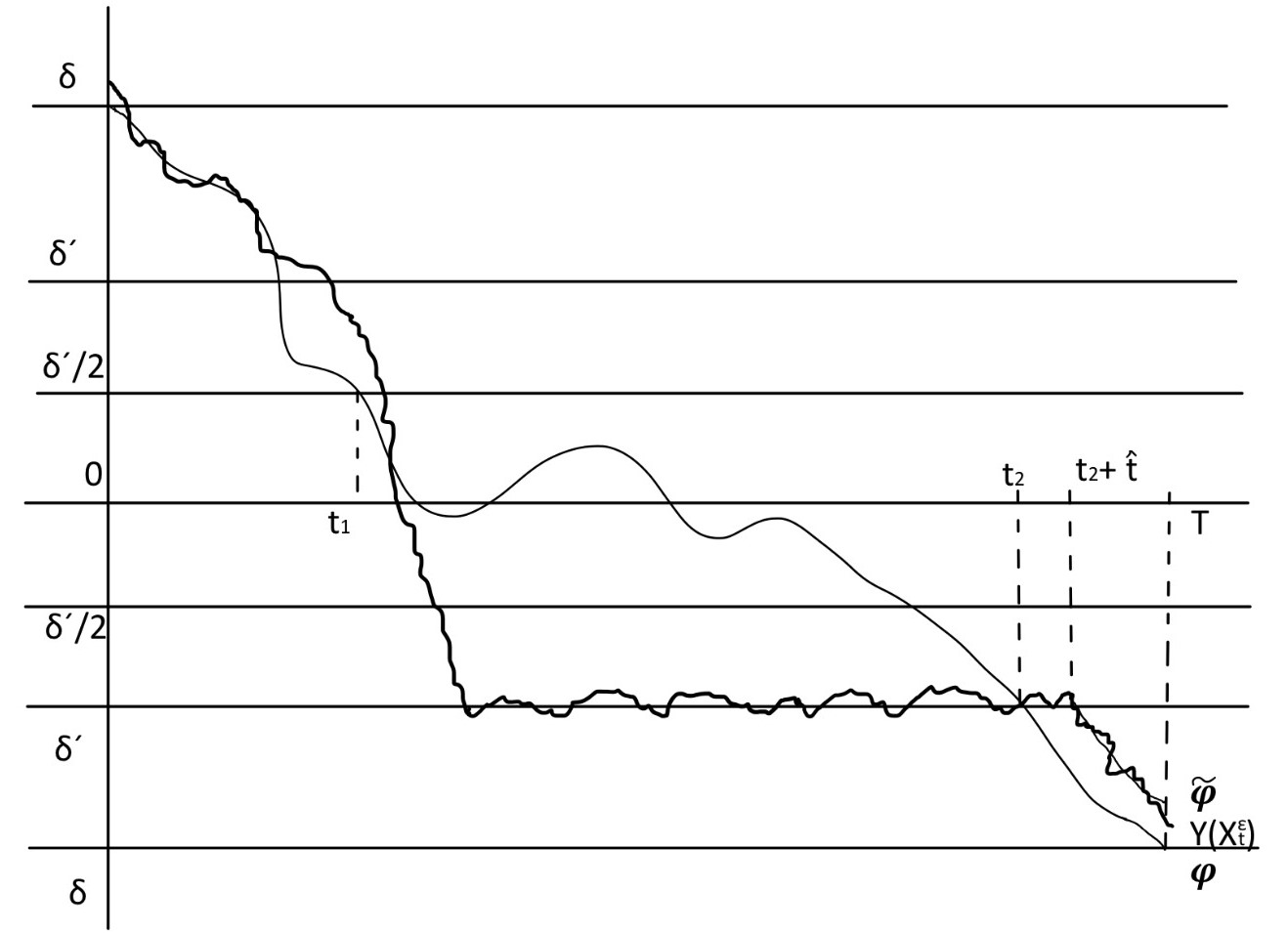}
    \caption{A typical graph where event $A$ happens.}
    \label{throughseparatrix}
\end{figure}
By Lemma~\ref{loclem}, we can fix $h_0<\delta$ such that, for any $x_0$ satisfying $|H(x_0)-\varphi_0|<h_0$,
\begin{equation}
\label{separatrix3eq:2}
{\bm{\mathrm P}}( \rho_{0,t_1}(H(X_t^\e),\varphi_t)<\delta'/2)\geq\exp\left[-\e^{-\beta}(S_{0,t_1}(\varphi)+\gamma/4)\right],
\end{equation}
and we can fix $\mu<\delta'$ such that, conditioned on $X^\e_{t_2+\hat t}=x$ satisfying $|H(x)-\Tilde\varphi_{t_2+\hat t}|<\mu$,
\begin{equation}
\label{separatrix3eq:3}
    {\bm{\mathrm P}}( \rho_{t_2+\hat t,T}(H(X_t^\e),\Tilde\varphi_t)<h/2)\geq\exp\left[-\e^{-\beta}(S_{t_2,T-\hat t}(\varphi)+\gamma/4)\right],
\end{equation} where $\Tilde\varphi_t=\varphi_{t-\hat t}$. Furthermore, let $\tau_1=\inf\{t:X_t^\e\in C_{kj}(\delta')\}$ and $\tau_2=\inf\{t:X_t^\e\in C_{ki}(\delta')\}$.
Let's consider the event:
$$A=\begin{Bmatrix}\rho_{0,t_1}(H(X_t^\e),\varphi_t)<\delta'/2;\tau_2-\tau_1<\hat t,\ X_t^\e\in D_k(\delta),\ t\in[\tau_1,\tau_2];\\
X_t^\e\in D_i(\delta'-\mu,\delta'+\mu), t\in[\tau_2,t_2+\hat t ];\ |H(X_t^\e)-\Tilde\varphi_t|<h/2, t\in[t_2+\hat t,T]\end{Bmatrix}.$$
Note that $A$ implies the desired event, hence it suffices to prove ${\bm{\mathrm P}}(A)\geq\exp(-\e^{-\beta}S_{0T}({\bm{\varphi}})+\gamma).$ To see this, combine \eqref{separatrix3eq:1}, \eqref{separatrix3eq:2} and \eqref{separatrix3eq:3} and note that \eqref{eq:4} implies that $${\bm{\mathrm P}}\left(X_t^\e\in D_i(\delta'-\mu,\delta'+\mu), t\in[\tau_2,t_2+\hat t ]\right)\geq\frac{1}{2},$$ for $\e$ small enough. Thus, by the strong Markov property of $X_t^\e$,
\begin{align*}
    {\bm{\mathrm P}}(A)\geq\exp(-\e^{-\beta}(S_{0T}({\bm{\varphi}})+\gamma)).
\end{align*}

\end{proof}

%% file: Upperbound.tex
\section{Upper bound for the Large Deviation Principle}
In this section, we deal with the upper bound for the local large deviation principle. We start with the discussion about what happens inside a single edge (Lemma~\ref{condition3} and Lemma~\ref{upperrate}). Again, we simplify the notations when there's no ambiguity, as we did in Section 3.1.

\begin{lemma} 
\label{condition3}
Suppose that $C=\overline{D_i(H_1,H_2)}\subset D_i$ and $X_0^\e=x\in C$. Then, for every $\delta>0$, every $\gamma>0$, and every $s>0$, there exists $\e_0>0$ such that
$${\bm{\mathrm P}}(X_t^\e\in C,\rho_{0,T}(Y(X_t^\e),\Phi(s))\geq\delta)\leq\exp(-\e^{-\beta}(s-\gamma))$$for all $\e\leq\e_0$ and every $x$, where $\Phi(s)=\{{\bm{\varphi}}\in \textbf{C}\left([0,T],I_i\right):S({\bm{\varphi}})\leq s,\  \varphi_0\in [H_1,H_2]\}$.
\end{lemma}
\begin{proof}
Without loss of generality, we let $T=1$. Since $H(C)$ is compact, and $C$ is away from the critical points, there exist $m>0$, $M>0$, and $L>0$, such that:
\begin{enumerate}
\item[(1)]$\lvert AH(x)\rvert+|\nabla H(x)^*\sigma(x)|^2<M,\ \forall x\in C$;
\item[(2)]$|\nabla H(x)^*\sigma(x)|^2>m,\ \forall x\in C$; 
\item[(3)]$m<T(H)<M\text{  and  }m<B^2(H)<M,\ \forall H\in H(C)$; 
\item[(4)]$\lvert\nabla H(x)-\nabla H(y)\rvert+\left|\lvert\nabla H(x)^*\sigma(x)\rvert^2-\lvert\nabla H(y)^*\sigma(y)\rvert^2\right|<L\lvert x-y\rvert,\ \forall x,y\in C$; 
\item[(5)]$\lvert T(H_1)-T(H_2)\rvert+\lvert T(H_1)B^2(H_1)-T(H_2)B^2(H_2)\rvert<L\lvert H_1-H_2\rvert,\ \forall H_1,H_2\in H(C)$.
\end{enumerate}
For $s>0$, $\gamma>0$, and $\delta>0$, we choose the following parameters:
\begin{enumerate}
\item[(1)]$0<\mu<1/3$ such that $s(1-2\mu)>s-\gamma$;\par 
\item[(2)]$\delta'>0$ such that $\delta'<\delta\wedge\frac{m^2\mu^2}{4L(2M+1)}$,
which particularly implies that
$$L\delta'(2M+1)<\frac{\mu m^2}{3},$$$$\frac{1-\mu^2}{2}\cdot \frac{2m+L\delta'}{2m-L\delta'}<\frac{1-\mu^2/2}{2}; $$\par 
\item[(3)]$n\in\mathbb N$ such that $n{\delta'}^2>64M(s-\gamma).$
\end{enumerate}

For every $0\leq k<n$, define the random variable $T_k^\e=\e^{1-\beta}T(H(X_\kk^\e))$, the event $A_k$ as in Lemma~\ref{loclem}, and the event $B_k$ where the process stays within the compact set, in which $|\nabla H(X_t^\e)|$ is bounded, so that $A_k$ happens with overwhelming probability:
\begin{align*}
A_k=&\bigcap_{p=0}^{\lfloor\frac{1}{nT^\e_k}\rfloor}\left\{\sup_{pT^\e_k+\kk\leq t\leq (p+1)T^\e_k+\kk}\e^{\beta/2}\left|\int_{k/n+pT^\e_k}^t \sigma(X_s^\e)dW_s\right|\leq M\cdot\e^{\frac{1-\beta}{4}}\right\}\\
&\bigcap\left\{\sup_{0\leq t\leq\frac{1}{n}}\left|H(X_{\kk+t}^\e)-H(X_\kk^\e)\right|\leq\frac{\delta'}{2}\right\};\\
B_k=&\left\{X_t^\e\in C,\forall t\in[\kk,\kl]\right\}.
\end{align*}
If we denote
$$Y_k=\frac{1}{2\e^\beta}\cdot\frac{n}{B^2\left(H(X_\kk^\e)\right)-\frac{L\delta'}{2}}\left|H(X_\kl^\e)-H(X_\kk^\e)\right|^2,$$
then, by Lemma \ref{finiteexpectation}, we have $\bm{\mathrm E}(\chi_{B_k\cap A_k}\cdot\exp((1-\mu)Y_k)\given)<c_\mu$ for some $0<c_\mu<\infty$.

Let $l^\e$ be the random polygon with vertices at $\left(\frac{k}{n},H(X_\kk^\e)\right)$ for $0\leq k\leq n$, and let $B=\bigcap_{k=1}^{n-1} B_k$. Then
\begin{align*}
    &{\bm{\mathrm P}}\left(X_t^\e\in C, \rho_{0,1}(Y(X_t^\e),\Phi(s))\geq\delta\right)\\
    =\ &{\bm{\mathrm P}}\left(X_t^\e\in C, \rho_{0,1}(Y(X_t^\e),\Phi(s))\geq\delta,\rho_{0,1}(Y(X_t^\e),l_t^\e)<\delta)\right)\\
    +&{\bm{\mathrm P}}\left(X_t^\e\in C, \rho_{0,1}(Y(X_t^\e),\Phi(s))\geq\delta,\rho_{0,1}(Y(X_t^\e),l_t^\e)\geq\delta)\right)\\
    \leq\ &{\bm{\mathrm P}}\left(X_t^\e\in C, S(l^\e)>s\right)+{\bm{\mathrm P}}\left(\left\{\sup_{0\leq t\leq 1}|H(X_t^\e)-l_t^\e|\geq\delta\right\}\cap B\right)\\
    \leq\ &{\bm{\mathrm P}}\left(\bigcap_{k=0}^{n-1} A_k\bigcap\{S(l^\e)>s\}\bigcap B\right)+{\bm{\mathrm P}}\left(\bigcup_{k=0}^{n-1}(\Omega\setminus A_k)\bigcap B\right)\\
    +&{\bm{\mathrm P}}\left(\left\{\sup_{0\leq t\leq 1}|H(X_t^\e)-l_t^\e|\geq\delta\right\}\bigcap B\right).
\end{align*}
We estimate the first term: 
\begin{align}
    &{\bm{\mathrm P}}\left(\bigcap_{k=0}^{n-1} A_k\bigcap\{S(l^\e)>s\}\bigcap B\right)\nonumber\\
    =\ &{\bm{\mathrm P}}\left(\bigcap_{k=0}^{n-1} \left(A_k\cap B_k\right)\bigcap\left\{\frac{1}{2}\int_0^1\frac{{\lvert {l^\e_t}^{'}\rvert}^2}{B^2(l_t^\e)}dt>s\right\}\right)\nonumber\\
    =\ &{\bm{\mathrm P}}\left(\bigcap_{k=0}^{n-1} \left(A_k\cap B_k\right)\bigcap\left\{\frac{n^2}{2}\sum_{k=0}^{n-1}\int_\kk^\kl\frac{\left|H(X_\kl^\e)-H(X_\kk^\e)\right|^2}{B^2(l_t^\e)}dt>s\right\}\right)\nonumber\\
    \leq\ &{\bm{\mathrm P}}\left(\bigcap_{k=0}^{n-1} \left(A_k\cap B_k\right)\bigcap\left\{\e^\beta \sum_{k=0}^{n-1}Y_k>s\right\}\right)\nonumber\\
    \leq\ &\frac{{\bm{\mathrm E}}\left[\prod_{k=0}^{n-1}\chi_{A_k\cap B_k}\cdot\exp\left((1-\mu)Y_k\right)\right]}{\exp\left((1-\mu)s\cdot\e^{-\beta}\right)}\label{cheby}\\
    =\ &\frac{{\bm{\mathrm E}}\left[\prod_{k=0}^{n-2}\chi_{A_k\cap B_k}\cdot\exp\left((1-\mu)Y_{k}\right)\cdot {\bm{\mathrm{E}}}\left(\chi_{B_{n-1}\cap A_{n-1}}\cdot\exp((1-\mu)Y_{n-1})\mathrel{\stretchto{\mid}{4ex}}\mathcal F_{\frac{n-1}{n}}\right)\right]}{\exp\left((1-\mu)s\cdot\e^{-\beta}\right)}\nonumber\\
    \leq\ &c_\mu\cdot\frac{{\bm{\mathrm E}}\left[\prod_{k=0}^{n-2}\chi_{A_k\cap B_k}\cdot\exp\left((1-\mu)Y_{k}\right)\right]}{\exp\left((1-\mu)s\cdot\e^{-\beta}\right)}\label{lemmaa9}\\
    \leq\ &\frac{c_\mu^n}{\exp\left((1-\mu)s\cdot\e^{-\beta}\right)}\label{induction}\\
    \leq\ &\frac{1}{3}\exp(-(1-2\mu)s\cdot\e^{-\beta})\nonumber\\
    \leq\ &\frac{1}{3}\exp(-\e^{-\beta}(s-\gamma))\nonumber,
\end{align}
where \eqref{cheby} follows from the exponential Chebyshev's inequality, and \eqref{lemmaa9} follows from Lemma~\ref{finiteexpectation}, and \eqref{induction} can be obtained by applying Lemma~\ref{finiteexpectation} repeatedly.
We estimate the second term using Lemma \ref{ak2}:
$${\bm{\mathrm P}}\left(\bigcup_{k=0}^{n-1}(\Omega\setminus A_k)\bigcap B\right)\leq n\exp\left(-\frac{n{\delta'}^2}{64M\e^{\beta}}\right)\leq\frac{1}{3}\exp(-\e^{-\beta}(s-\gamma)),$$
by recalling that $n{\delta'}^2>64M(s-\gamma)$. 
Similarly, we estimate the third term:
\begin{align*}
    {\bm{\mathrm P}}\left(\left\{\sup_{0\leq t\leq 1}|H(X_t^\e)-l_t^\e|\geq\delta\right\}\bigcap B\right)\leq&\ \sum_{k=0}^{n-1}{\bm{\mathrm P}}\left(\left\{\sup_{\kk\leq t\leq\kl}|H(X_t^\e)-l_t^\e|\geq\delta\right\}\bigcap B\right)\\
    \leq&\ \sum_{k=0}^{n-1}{\bm{\mathrm P}}\left(\left\{\sup_{0\leq t\leq \frac{1}{n}}\left|H(X_{\kk+t}^\e)-H(X_\kk^\e)\right|>\frac{\delta'}{2}\right\}\bigcap B\right)\\
    \leq&\ n\exp\left(-\frac{n{\delta'}^2}{64M\e^{\beta}}\right)\\
    \leq&\ \frac{1}{3}\exp(-\e^{-\beta}(s-\gamma)).
\end{align*}

Thus $${\bm{\mathrm P}}(X_t^\e\in C,\rho_{0,1}(Y(X_t^\e),\Phi(s))\geq\delta)\leq\exp(-\e^{-\beta}(s-\gamma)).$$
\end{proof}
The next result, which relies on Lemma~\ref{condition3} and Theorem~\ref{compactness}, gives an upper bound on probability that the process stays close to a given trajectory on the graph.
\begin{lemma} 
\label{upperrate}
Suppose  that ${\bm{\varphi}}=(i,\varphi)\in\textbf{C}\left([0,T],I_i\right)$ is such that $\inf_{t\in[0,T]}r(\varphi_t,H(O))>0$ for each $O\sim I_i$ and $S({\bm{\varphi}})<\infty$. Then, for every $\gamma>0$,  there exist $\delta>0$ and $\e_0>0$ such that, for each $\e\leq\e_0$ and each initial point $x_0\in D_i(\varphi_0-\delta,\varphi_0+\delta)$, 
$${\bm{\mathrm P}}\left(\rho_{0,T}(Y(X_t^\e),{\bm{\varphi}}\right)\leq\delta)\leq\exp\left(-\e^{-\beta}(S({\bm{\varphi}})-\gamma)\right).$$
\end{lemma}
\begin{proof}
Without loss of generality, we assume that $O_j\sim I_i$, $O_k\sim I_i$, and $H(\bm{\mathrm x}_k)>H(\bm{\mathrm x}_j)=0$.
Suppose that $0<m'\leq\varphi\leq M'<H(\bm{\mathrm x}_k)$. We define the compact set $C=\overline{D_i(m'/2,\frac{M'+H(\bm{\mathrm x}_k)}{2})}\subset D_i$ and $s=S({\bm{\varphi}})-\gamma/2$. 
Then, by Theorem~\ref{compactness}, $\rho_{0,T}({\bm{\varphi}},\Phi(s))>0$, where $\Phi(s)=\{\Tilde{\bm{\varphi}}:S(\Tilde{\bm{\varphi}})\leq s,\ \Tilde\varphi_0\in [m'/2,\frac{M'+H(\bm{\mathrm  x}_k)}{2}]\}$ (if $\Phi(s)$ is empty, then $\rho_{0,T}({\bm{\varphi}},\Phi(s))=\infty$). 
We choose $\delta>0$ such that $\delta<\rho_{0,T}({\bm{\varphi}},\Phi(s))/2$ and $D_i(m'-\delta,M'+\delta)\subset C$. 
If $\rho_{0,T}(Y(X_t^\e),{\bm{\varphi}})<\delta$, then
$$\rho_{0,T}(Y(X_t^\e),\Phi(s))\geq\rho_{0,T}({\bm{\varphi}},\Phi(s))-\rho_{0,T}(Y(X_t^\e),{\bm{\varphi}})>\delta.$$
So, by Lemma~\ref{condition3}, for $\e$ small enough,
\begin{align*}
    {\bm{\mathrm P}}\left(\rho_{0,T}(Y(X_t^\e),{\bm{\varphi}})\right)<\delta)=&\ {\bm{\mathrm P}}\left(\rho_{0,T}(Y(X_t^\e),{\bm{\varphi}}\right)<\delta, X_t^\e\in C)\\
    \leq&\ {\bm{\mathrm P}}\left(\rho_{0,T}(Y(X_t^\e),\Phi(s))>\delta, X_t^\e\in C\right)\\
    \leq&\ \exp\left(-\e^{-\beta}(s-\gamma/2)\right)\\
    =&\ \exp\left(-\e^{-\beta}(S({\bm{\varphi}})-\gamma)\right).
\end{align*}
\end{proof}

\begin{theorem} 
\label{upperbound}
Given ${\bm{\varphi}}\in\textbf{C}\left([0,T],\Gamma\right)$ with $\bm{\varphi_0}=Y(x_0)$,
$$\varlimsup_{\delta\to0}\varlimsup_{\e\to0}\e^\beta\log {\bm{\mathrm P}}\left(\rho_{0,T}(Y(X_t^\e),{\bm{\varphi}}))<\delta\right)\leq-S({\bm{\varphi}}).$$
\end{theorem}
\begin{proof}
There exists a compact set $C$ such that $\{x:\rho_{0,T}(Y(x),{\bm{\varphi}})\leq\delta\}\subset C$. 
Let $\overline{M}$ be such that $$\lvert AH(x)\rvert+|\nabla H(x)^*\sigma(x)|^2<\overline M,\ \forall x\in C.$$
Let us consider the case where neither ${\bm{\varphi}}_0$ nor ${\bm{\varphi}}_T$ is at a vertex. Fix $\gamma>0$. If $S({\bm{\varphi}})=\frac{1}{2}\int_0^T\frac{|\dot\varphi_t|^2}{B_{i_t}^2(\varphi_t)}dt<\infty$, then there exists $0<\delta'<\inf_k|\varphi_0-H(\bm{\mathrm  x}_k)|\wedge\inf_k|\varphi_T-H(\bm{\mathrm  x}_k)|$ so small that
$$\int_0^T\frac{|\dot\varphi_t|^2}{B_{i_t}^2(\varphi_t)}\cdot\chi_{\{{\bm{\varphi}}_t\in D(\pm\delta')\}}dt<\gamma.$$ 
Now let us define three sequences in $[0,T]$ inductively:
\begin{align*}
    q_0&=0;\\
    t_k&=\inf\left\{t>q_{k-1}:r({\bm{\varphi}}_t,O)=\delta'\text{ for some vertex }O\right\},\\
    p_k&=\sup\left\{t<t_k:r({\bm{\varphi}}_t,O)=2\delta'\text{ for some vertex }O\right\},\\
    q_k&=\inf\left\{t>t_k:r({\bm{\varphi}}_t,O)=2\delta'\text{ for some vertex }O\right\},
\end{align*}
for $1\leq k\leq n$, and $n$ is such that the set $\left\{t>q_n:{\bm{\varphi}}_t\in D(\delta')\right\}$ is empty. Define $p_{n+1}=T$. The sequences are all finite due to the absolutely continuity of $\varphi$. Hence, by Lemma~\ref{upperrate}, there exist $\delta>0$ and $\e_0>0$ such that, 
conditioned on $X_{q_k}^\e=x$, 
for all $\e\leq\e_0$,
$${\bm{\mathrm P}}\left(\rho_{q_k,p_{k+1}}(Y(X_t^\e),{\bm{\varphi}}\right)\leq\delta)\leq\exp\left(-\e^{-\beta}(\frac{1}{2}\int_{q_k}^{p_{k+1}}\frac{|\dot\varphi_t|^2}{B_{i_t}^2(\varphi_t)}dt-2^{-k}\cdot\gamma)\right),$$ 
since $\{t:r({\bm{\varphi}}_t,O)\geq\delta'\text{ for every vertex $O$ on }\Gamma\}\supset\bigcup_{k=0}^n(q_k,p_{k+1})$. Thus we obtain
\begin{align*}
    {\bm{\mathrm P}}\left(\rho_{0,T}(Y(X_t^\e),{\bm{\varphi}}))<\delta\right)\leq&\ {\bm{\mathrm P}}\left(\rho_{q_k,p_{k+1}}(H(X_t^\e),\varphi)<\delta,\ 0\leq k\leq n\right)\\
    \leq&\ \exp\left(-\e^{-\beta}(\frac{1}{2}\int_{\bigcup_{k=0}^n(q_k,p_{k+1})}\frac{|\dot\varphi_t|^2}{B_{i_t}^2(\varphi_t)}dt-2\gamma)\right)\\
    \leq&\ \exp\left(-\e^{-\beta}(S({\bm{\varphi}})-\gamma-2\gamma)\right)\\
    =&\ \exp\left(-\e^{-\beta}(S({\bm{\varphi}})-3\gamma)\right)
\end{align*}
by the strong Markov property. Since we chose $\gamma>0$ arbitrarily,$$
\varlimsup_{\delta\to0}\varlimsup_{\e\to0}\e^\beta\log {\bm{\mathrm P}}\left(\rho_{0,T}(Y(X_t^\e),{\bm{\varphi}}))<\delta\right)\leq-S({\bm{\varphi}}).
$$
In the case where $S({\bm{\varphi}})=\infty$, it is easy to show that, for every $s>0$, there exist finitely many disjoint closed intervals $\{J_k\}$ such that $\bm{\varphi}$ stays away from all vertices on all $J_k$'s and $\sum_{k} S_{J_k}(\bm{\varphi})>s$. The result then follows in a similar way as in the finite action functional case.
\end{proof}

%% file: ExponentialTighiness.tex
\section{Exponential Tightness}
In this section, we establish the exponential tightness of the process $Y(X_t^\e)$ in $\textbf{C}\left([0,T],\Gamma\right)$. We start with the definition before heading to the proof (cf. \cite{Liptser1996LargeDF}).
\begin{definition}
    The family $Y(X_t^\e)$ is exponentially tight (with rate $\e^\beta$) in $\textbf{C}\left([0,T],\Gamma\right)$ if there exists an increasing sequence of compact sets $(K_j)_{\{j\geq1\}}$ in $\textbf{C}\left([0,T],\Gamma\right)$ such that 
    $$\lim_{j\to\infty}\varlimsup_{\e\to 0}\e^\beta\log\bm{{\bm{\mathrm P}}}(Y(X_t^\e)\in \textbf{C}\left([0,T],\Gamma\right)\setminus K_j)=-\infty.$$
\end{definition}Our approach is based on Theorem 3.1 in \cite{semi}. It suffices to check the following conditions:
\begin{align*}
    &(\romannumeral1)\lim_{C\to\infty}\varlimsup_{\e\to0}\e^\beta\log {\bm{\mathrm P}}\left(\sup_{[0,T]}|H(X_t^\e)|>C\right)=-\infty,\\
    &(\romannumeral2)\lim_{\delta\to0}\varlimsup_{\e\to0}\e^\beta\log\sup_{\tau\leq T-\delta}{\bm{\mathrm P}}\left(\sup_{t\in [0,\delta]}r(Y(X_{\tau+t}^ \e),Y(X_\tau^\e))>\eta\right)=-\infty,\ \ \ \forall\eta>0,
\end{align*}
where $\tau$ is a stopping time w.r.t. the filtration $\mathcal F$.
\begin{theorem}
\label{exponentialtightness}
Under our assumptions, conditions $(\romannumeral1)$ and $(\romannumeral2)$ are both valid, hence the family $Y(X^\e_t)$ is exponentially tight (with rate $\e^\beta)$ in $\textbf{C}\left([0,T],\Gamma\right)$.
\end{theorem}
\begin{proof}
We present the proof in two steps. \textbf{\romannumeral1.} Let us check condition $(\romannumeral1)$ first. 
By assumptions, there exists $M$ large such that $H(x)\geq A_1|x|^2$ for $|x|\geq M$, $|AH(x)|+\lVert\sigma(x)\rVert\leq M$ for $x\in\mathbb R^2$, and $|\nabla H|$ is Lipschitz with constant $M$.
For any large $C>A_1M^2+2|H(x_0)|+A_1|x_0|^2+1$, define the stopping time $\tau=\inf\{t:H(X_t^\e)=C\}\wedge T$. Then we have
$$H(X_\tau^\e)=H(x_0)+\e^\beta\int_0^\tau AH(X_s^\e)ds+\e^{\beta/2}\int_0^\tau\rcoe dW_s.$$
Let $F=\{x:H(x)\leq C\}$. Then, for $x\in F$, $|x|\leq\sqrt{\frac{C}{A_1}}$.
In light of the boundedness of $\sigma$ and of the second derivatives of $H$,
\begin{equation}
    \label{eq:6.1}
    |\nabla H^*(x)\sigma(x)|\leq M^2|x-x_0|<M^2(\sqrt{\frac{C}{A_1}}+|x_0|)<2M^2\sqrt{\frac{C}{A_1}},
\end{equation}
for each $x\in F$.
From the fact $H(X_\tau^\e)=C$ we deduce that, on the event $\{\tau< T\}$, 
\begin{align*}
    \e^{\beta/2}\int_0^\tau\rcoe dW_s=&\ C-H(x_0)-\e^\beta\int_0^\tau AH(X_s^\e)ds\\
    \geq&\ C-H(x_0)-MT\e^\beta\\
    \geq&\ \frac{1}{2}C,
\end{align*}
for $\e$ sufficiently small. Thus, on the event $\{\tau<T\}$,
\begin{align*}
    \int_0^T\chi_{\{\tau\geq s\}}\rcoe dW_s\geq&\ \frac{1}{2}C\e^{-\beta/2}\\
    \Tilde W(\int_0^T|\chi_{\{\tau\geq s\}}\rcoe|^2ds)\geq&\ \frac{1}{2}C\e^{-\beta/2}\\
    \sup_{[0,\frac{4M^4TC}{A_1}]}\Tilde W_t\geq&\ \frac{1}{2}C\e^{-\beta/2},
\end{align*}
where $\Tilde W$ is another Brownian motion. Hence $\{\tau< T\}$ implies that $\{\sup_{t\in[0,\frac{4M^4TC}{A_1}]}\Tilde W_t\geq\ \frac{1}{2}C\e^{-\beta/2}\}$.
As for the probability ${\bm{\mathrm P}}\left(\sup_{[0,T]}|H(X_t^\e)|>C\right)$ for fixed $C$ large enough,
\begin{align*}
    {\bm{\mathrm P}}\left(\sup_{[0,T]}|H(X_t^\e)|>C\right)=&\ {\bm{\mathrm P}}\left(\tau<T\right)\\
    \leq&\ {\bm{\mathrm P}}\left(\sup_{[0,\frac{4M^4TC}{A}]}\Tilde W_t\geq\frac{1}{2}C\e^{-\beta/2}\right)\\
    =&\ 2{\bm{\mathrm P}}\left(\Tilde W_{\frac{4M^4TC}{A}}\geq\frac{1}{2}C\e^{-\beta/2}\right)\\
    \leq&\ \sqrt{\frac{32M^4T\e^\beta}{\pi AC}}\exp(-\frac{AC}{32M^4T\e^\beta}).
\end{align*}
Therefore, $$\lim_{C\to\infty}\varlimsup_{\e\to0}\e^\beta\log {\bm{\mathrm P}}\left(\sup_{[0,T]}|H(X_t^\e)|>C\right)=-\infty$$

\textbf{\romannumeral2 .} Since we've already checked the validity of the first condition, it suffices to instead prove that for every $\eta>0$, $C>0$,
$$\lim_{\delta\to0}\varlimsup_{\e\to0}\e^\beta\log\sup_{\tau\leq T-\delta}{\bm{\mathrm P}}\left(\sup_{t\in [0,\delta]}r(Y(X_{\tau+t}^ \e),Y(X_\tau^\e))>\eta,\sup_{[0,T]}|H(X_t^\e)|\leq C\right)=-\infty.$$
Without loss of generality, we assume that $\eta$ is so small that $\eta<|H(\bm{\mathrm x}_k)-H(\bm{\mathrm x}_j)|$ for any critical points $\bm{\mathrm x}_j\not=\bm{\mathrm x}_k$. Note that  
$$    \left\{\sup_{t\in [0,\delta]}r(Y(X_{\tau+t}^ \e),Y(X_\tau^\e))>\eta\right\}\subset\left\{\sup_{t\in [0,\delta]}|H(X_{\tau+t}^ \e)-H(X_\tau^\e)|>\eta/3\right\}. $$
Thus it suffices that show that, for any positive $\eta$ and $C$,
$$\lim_{\delta\to0}\varlimsup_{\e\to0}\e^\beta\log\sup_{\tau\leq T-\delta}{\bm{\mathrm P}}\left(\sup_{t\in [0,\delta]}|H(X_{\tau+t}^ \e)-H(X_\tau^\e)|>\frac{\eta}{3},\sup_{[0,T]}|H(X_t^\e)|\leq C\right)=-\infty.$$
Since $AH(x)$ is bounded, this is implied by
$$\lim_{\delta\to0}\varlimsup_{\e\to0}\e^\beta\log\sup_{\tau\leq T-\delta}{\bm{\mathrm P}}\left(\sup_{t\in [0,\delta]}|\int_{\tau}^{\tau+t}\rcoe dW_s|>\frac{\eta}{4},\sup_{[0,T]}|H(X_t^\e)|\leq C\right)=-\infty.$$
Define the process $Y_t^\e=\e^{\beta/2}\int_0^{t}\rcoe dW_s$, which is a martingale.
Fix $\delta>0$. Then, for any stopping time $\tau\leq T-\delta$, define a random change of time $\hat\tau_t:=\tau+t$ and a new process $\hat Y_t^\e:=Y_{\hat\tau_t}$ with a new filtration $\hat{\mathcal {F}_t}:=\mathcal F_{\hat\tau_t}$, $\forall t\geq0$. 
By Theorem 1 from Chap.4, Sect.7 in \cite{martingale}, $\hat Y_t^\e$ is also a martingale, hence $Z_t^\e=\hat Y_t^\e-\hat Y_0^\e$ is martingale with filtration $\hat{\mathcal {F}_t}$. 
Furthermore, by Problem 2.28 from Chap.3 in \cite{shreve}, $\zeta_t^\e:=\exp(\lambda Z_t^\e-\frac{1}{2}\lambda^2\langle Z^\e\rangle_t)$ is supermartingale for each $\lambda\in\mathbb R$. 
Take a stopping time $\sigma:=\inf\{t\leq\delta:|Z_t^\e|\geq\eta/4\}\wedge \delta$ w.r.t. $\hat{\mathcal {F}_t}$. 
By the Optional Sampling Theorem, $E\zeta_\sigma^\e\chi_{\{Z_\sigma^\e\geq\eta/4,\ \sup_{[0,T]}|H(X_t^\e)|\leq C\}}\leq E\zeta_0^\e=1$. 
Keeping \eqref{eq:6.1} in mind, we conclude that 
$${\bm{\mathrm P}}\left(Z_\sigma^\e\geq\eta/4,\ \sup_{[0,T]}|H(X_t^\e)|\leq C\right)\leq\exp\left(-\frac{\lambda\eta}{4}+\frac{\lambda^2}{2}\cdot\frac{4M^4C}{A_1}\cdot\delta\e^\beta\right)$$holds for each $\lambda\in\mathbb R$. Thus, ${\bm{\mathrm P}}\left(Z_\sigma^\e\geq\eta/4,\ \sup_{[0,T]}|H(X_t^\e)|\leq C\right)\leq\exp\left(-\frac{A_1\eta^2}{128M^4C\delta\e^\beta}\right)$. Similarly, one can prove ${\bm{\mathrm P}}\left(Z_\sigma^\e\leq-\eta/4,\ \sup_{[0,T]}|H(X_t^\e)|\leq C\right)\leq\exp\left(-\frac{A_1\eta^2}{128M^4C\delta\e^\beta}\right)$. So the second condition is also valid.
\end{proof}

%% file: proof.tex
\section{Proof of the Main Result}
The following theorem is a result by Puhalskii adapted to our particular case (cf. Theorem 1.2 in \cite{semi}, and also Corollary 3.4 in \cite{Puhalskii}).
\begin{theorem}
\label{puhalskii}
    If the family of process $Y(X_t^\e)$ is exponentially tight (with rate $\e^\beta$) and functional $S$ satisfies the following condition for every $\bm{\varphi}$ with $\bm{\varphi}_0=Y(x_0)$:
\begin{align}
    S(\bm{\varphi})=&-\lim_{\delta\to0}\varlimsup_{\e\to0}\e^\beta\log {\bm{\mathrm P}}\left(\rho_{0,T}(Y(X_t^\e),\bm{\varphi}))<\delta\right)\nonumber\\
    =&-\lim_{\delta\to0}\varliminf_{\e\to0}\e^\beta\log {\bm{\mathrm P}}\left(\rho_{0,T}(Y(X_t^\e),\bm{\varphi}))<\delta\right),\label{locallargedev}
\end{align}
then $\e^{-\beta}S(\bm\varphi)$ is the action functional of the family of process $Y(X_t^\e)$ in $\textbf{C}\left([0,T],\Gamma\right)$ restricted to the set of functions that start at $Y(x_0)$.
\end{theorem}
\begin{proof}[Proof of Theorem~\ref{main}]
    The validity of the conclusion is due to Theorem~\ref{lowermaintheorem}, Theorem~\ref{upperbound}, Theorem~\ref{exponentialtightness}, and Theorem~\ref{puhalskii}.
\end{proof}

%% file: appendix.tex
\appendix
\section{Appendix}

\textbf{Part 1.} We prove the level set $\Phi_\delta(s)=\left\{r(\bm{\varphi_0},Y(x_0))\leq\delta:S(\varphi)\leq s\right\}$ is compact in $\textbf{C}\left([0,T],\mathbb R\right)$ for any $\delta\geq0$, $s\geq0$.
\input{compactness}

\textbf{Part 2.} We make two claims about the behavior of the random process during a short period of time.

\begin{lemma}
Let the assumptions of Lemma \ref{lowerbound} be satisfied and event $A_k$ be defined as in the proof of Lemma~\ref{lowerbound}. Then, for $\e$ sufficiently small,$${\bm{\mathrm P}}\left(\Omega\setminus A_k\right)\leq\exp\left(-\frac{n{\delta'}^2}{64M\e^{\beta}}\right),$$
for every initial point $X_k^\e=x$ that satisfies $|H(x)-\varphi_\kk|<h$.
\label{ak}
\end{lemma}
\begin{proof}
Without loss of generality, we assume that $k=0$ in the proof. 
On the interval $[0,\frac{1}{n}]$, we have $|AH(X_t^\e)|<M$, so the increment in $H(X_t^\e)$ due to the ordinary integral is negligible in the sense that
\begin{align*}
    |H(X_{t}^\e)-H(X_0^\e)|\leq&\ \e^\beta\int_0^{t}|AH(X_s^\e)|ds+\e^{\beta/2}\left|\int_0^{t}\rcoe dW_s\right|\\
    \leq&\ \e^\beta\cdot\frac{M}{n}+\e^{\beta/2}\left|\int_0^{t}\rcoe dW_s\right|.
\end{align*}
Hence by recalling that, in Lemma~\ref{lowerbound}, $C$ was defined as a compact set contained in $D_i$, for $\e$ small enough, we have that
\begin{align*}
    &\ {\bm{\mathrm P}}(\sup_{0\leq t\leq \frac{1}{n}}\left|H(X_{t}^\e)-H(X_0^\e)\right|\leq\frac{\delta'}{2})\\
    \geq&\ {\bm{\mathrm P}}(\sup_{0\leq t\leq \frac{1}{n}}\e^{\beta/2}\left|W^{(1)}(\int_0^{t}\left|\rcoe\right|^2ds\right|)\leq\frac{\delta'}{4},X_t^\e\in C)\\
    \geq&\ {\bm{\mathrm P}}(\sup_{0\leq t\leq\frac{M}{n}}\lvert W^{(1)}_t\rvert\leq\frac{\delta'}{4\e^{\beta/2}},X_t^\e\in C)\\
    =&\ {\bm{\mathrm P}}(\sup_{0\leq t\leq\frac{M}{n}}\lvert W^{(1)}_t\rvert\leq\frac{\delta'}{4\e^{\beta/2}})\\
    \geq&\ 1-\frac{8\sqrt{2M\e^\beta}}{\sqrt{\pi n}\cdot\delta'}\exp(-\frac{n{\delta'}^2}{32M\e^\beta})=:1-\lambda_1(\e),
\end{align*} where $W_t^{(1)}$ is a one-dimensional Wiener process. Equivalently, we have
$${\bm{\mathrm P}}\left(\sup_{0\leq t\leq \frac{1}{n}}\left|H(X_{t}^\e)-H(X_0^\e)\right|>\frac{\delta'}{2}\right)\leq\lambda_1(\e).$$

Let $\sigma_i$ denote the $i$-th row of $\sigma$. For $\e$ small enough,
\begin{align*}
    &{\bm{\mathrm P}}\left(\sup_{pT^\e_0\leq t\leq (p+1)T^\e_0}\e^{\beta/2}\left|\int_{pT^\e_0}^t \sigma(X_s^\e)dW_s\right|>M\cdot\e^{\frac{1-\beta}{4}}\right)\\
    \leq\ &{\bm{\mathrm P}}\left(\bigcup_{j=1}^{2}\{\sup_{pT^\e_0\leq t\leq (p+1)T^\e_0}\e^{\beta/2}\left|\int_{pT^\e_0}^t \sigma_j(X_s^\e)dW_s\right|>\frac{M\cdot\e^{\frac{1-\beta}{4}}}{\sqrt{2}}\}\right)\\
    \leq\ &\sum_{j=1}^{2}{\bm{\mathrm P}}\left(\sup_{pT^\e_0\leq t\leq (p+1)T^\e_0}\e^{\beta/2}\left|\int_{pT^\e_0}^t \sigma_j(X_s^\e)dW_s\right|>\frac{M\cdot\e^{\frac{1-\beta}{4}}}{\sqrt{2}}\right)\\
    \leq\ &\sum_{j=1}^{2}{\bm{\mathrm P}}\left(\sup_{pT^\e_0\leq t\leq (p+1)T^\e_0}\left|W^{(2)}_j(\e^\beta\int_{pT^\e_0}^t{\lvert\sigma_j(X_s^\e)\rvert}^2ds)\right|>\frac{M\cdot\e^{\frac{1-\beta}{4}}}{\sqrt{2}}\right)\\
    \leq\ &\sum_{j=1}^{2}{\bm{\mathrm P}}\left(\sup_{0\leq t\leq M^2T^\e_0\e^\beta}\left|W^{(2)}_j(t)\right|>\frac{M\cdot\e^{\frac{1-\beta}{4}}}{\sqrt{2}}\right)\\
    \leq\ &4\sum_{j=1}^{2}{\bm{\mathrm P}}\left(W_j^{(2)}(M^2T^\e_0\e^\beta)>\frac{M\cdot\e^{\frac{1-\beta}{4}}}{\sqrt{2}}\right)\\
    \leq\ &4\sum_{j=1}^{2}\frac{\sqrt{T(H(X_0^\e))}\cdot\e^{\frac{1+\beta}{4}}}{\sqrt\pi}\exp\left(-\frac{1}{4T(H(X_0^\e))\e^{\frac{1+\beta}{2}}}\right)\\
    =\ &\frac{8\sqrt{T(H(X_0^\e))}\cdot\e^{\frac{1+\beta}{4}}}{\sqrt\pi}\exp\left(-\frac{1}{4T(H(X_0^\e))\e^{\frac{1+\beta}{2}}}\right),\\
\end{align*}
where $W_j^{(2)}$ is a one-dimensional Wiener process for each $j$. Particularly, we conclude that
\begin{align*}
\ &{\bm{\mathrm P}}\left(\sup_{pT^\e_0\leq t\leq (p+1)T^\e_0}\e^{\beta/2}\left|\int_{pT^\e_0}^t \sigma(X_s^\e)dW_s\right|>M\cdot\e^{\frac{1-\beta}{4}}\right)\\
\leq\ &\frac{8\sqrt{T(H(X_0^\e))}\cdot\e^{\frac{1+\beta}{4}}}{\sqrt\pi}\exp\left(-\frac{1}{4T(H(X_0^\e))\e^{\frac{1+\beta}{2}}}\right)\\
\leq\ &\frac{8\sqrt{M}\cdot\e^{\frac{1+\beta}{4}}}{\sqrt\pi}\exp\left(-\frac{1}{4m\e^{\frac{1+\beta}{2}}}\right)=:\lambda_2(\e)<\lambda_1(\e).
\end{align*}
Since $|H(x)-\varphi_0|<h$, $T^\e_0\geq m\e^{1-\beta}$, by the discussions above, we know that the event $A_0$ happens with overwhelming probability when $\e$ is small enough, in the sense that
\begin{align*}
    &{\bm{\mathrm P}}\left(\Omega\setminus A_0\right)\\
    \leq\ &\lambda_1(\e)+{\bm{\mathrm P}}\left(\bigcup_{p=0}^{\lfloor\frac{1}{nT^\e_0}\rfloor}\sup_{pT^\e_0\leq t\leq (p+1)T^\e_0}\e^{\beta/2}\left|\int_{pT^\e_0}^t \sigma(X_s^\e)dW_s\right|>M\cdot\e^{\frac{1-\beta}{4}}\right)\\
    \leq\ &\lambda_1(\e)+{\bm{\mathrm P}}\left(\bigcup_{p=0}^{\lfloor\frac{1}{mn\e^{1-\beta}}\rfloor}\sup_{pT^\e_0\leq t\leq (p+1)T^\e_0}\e^{\beta/2}\left|\int_{pT^\e_0}^t \sigma(X_s^\e)dW_s\right|>M\cdot\e^{\frac{1-\beta}{4}}\right)\\
    \leq\ &\frac{2}{mn\e^{1-\beta}}\cdot\lambda_1(\e)\\
    =\ &\frac{2}{mn\e^{1-\beta}}\cdot\frac{8\sqrt{2M\e^\beta}}{\sqrt{\pi n}\cdot\delta'}\exp\left(-\frac{n{\delta'}^2}{32M\e^\beta}\right)\\
    \leq\ &\exp\left(-\frac{n{\delta'}^2}{64M\e^{\beta}}\right).
\end{align*}
\end{proof}

\begin{lemma}
Let the assumptions in Lemma~\ref{condition3} be satisfied and events $A_k$, $B_k$ be defined as in the proof of Lemma~\ref{condition3}.
Then, for $\e$ sufficiently small,$${\bm{\mathrm P}}\left((\Omega\setminus A_k)\cap B_k\right)\leq\exp\left(-\frac{n{\delta'}^2}{64M\e^{\beta}}\right).$$
\label{ak2}
\end{lemma}
\begin{proof}
The proof is almost the same as that of Lemma \ref{ak}.
\end{proof}

\textbf{Part 3.} We formulate the next two lemmas to prepare for the proof of Lemma~\ref{extreme}.
Lemma~\ref{positivedrift} shows that the random process that starts at an exterior vertex does not get stuck there; and Lemma~\ref{zeroderivative} shows that the trajectory that starts at an exterior vertex does not escape too fast. Recall the assumptions of Section 3.2: $\bm{\mathrm  x}_0=(0,0)$ (the origin on $\mathbb R^2$) is a local minimum point and $H(\bm{\mathrm  x}_0)=0$.

\begin{lemma}
With the same assumptions as in Lemma \ref{extreme}, in a small neighborhood $B(\bm{\mathrm x}_0,r)$, we have that, for every $x\in B(\bm{\mathrm x}_0,r)\setminus\{\bm{\mathrm x}_0\}$,
$$\frac{AH(x)}{2\sqrt{H(x)}}-\frac{|\nabla H^*(x)\sigma(x)|^2}{8\sqrt{H(x)^3}}>0.$$
\label{positivedrift}
\end{lemma}
\begin{proof}
It suffices to prove that $$4H(x)AH(x)-|\nabla H^*(x)\sigma(x)|^2> 0$$ in $B(\bm{\mathrm x}_0,r)$. By Taylor's expansion, in $B(\bm{\mathrm x}_0,r)$ we have $H(x)=a\h x^2+b\h y^2+2c\h x\h y+O(|x|^3)$, where $x=(\h x,\h y)$. The Hessian matrix of $H$ at $\bm{\mathrm x}_0$ is positive-definite, hence $a,b>0$, and $ab-c^2>0$. Since $H$ is four times continuously differentiable, the derivatives can also be approximated in the following way:
\begin{align*}
    \frac{\partial H}{\partial\h x}&=2a\h x +2c\h y+O(|x|^2),\ \ \ \ \ \ \ \ \
    \frac{\partial H}{\partial\h y}=2b\h y +2c\h x+O(|x|^2),\\
    \frac{\partial^2 H}{\partial\h x^2}&=2a+ O(|x|),\ \ \ \ \ 
    \frac{\partial^2 H}{\partial\h y^2}=2b+ O(|x|),\ \ \ \ \
    \frac{\partial^2 H}{\partial\h x\partial\h y}=2c+ O(|x|).
\end{align*}
Define the matrix-valued function 
$$b(x)=
\begin{bmatrix}
2H\frac{\partial^2H}{\partial\h x^2}-\left(\frac{\partial H}{\partial\h x}\right)^2 & 2H\frac{\partial^2H}{\partial\h x\partial\h y}-\frac{\partial H}{\partial\h x}\frac{\partial H}{\partial\h y}\\
2H\frac{\partial^2H}{\partial\h x\partial\h y}-\frac{\partial H}{\partial\h x}\frac{\partial H}{\partial\h y} & 2H\frac{\partial^2H}{\partial\h y^2}-\left(\frac{\partial H}{\partial\h y}\right)^2
\end{bmatrix}.
$$
With the Taylor's expansions, we can compute, for $x\in B(\bm{\mathrm x}_0,r)$ with $r$ sufficiently small,
\begin{align*}
    &2H\frac{\partial^2H}{\partial\h x^2}-\left(\frac{\partial H}{\partial\h x}\right)^2=4\h y^2(ab-c^2)+O(|x|^3)>0,\\
    &2H\frac{\partial^2H}{\partial\h y^2}-\left(\frac{\partial H}{\partial\h y}\right)^2=4\h x^2(ab-c^2)+O(|x|^3)>0,\\
    &2H\frac{\partial^2H}{\partial\h x\partial\h y}-\frac{\partial H}{\partial\h x}\frac{\partial H}{\partial\h y}=4\h x\h y(c^2-ab)+O(|x|^3).
\end{align*}
It follows that $\mathrm{det}\left(b(x)\right)=O(|x|^5)$. Since $a(x)=\sigma(x)\sigma^*(x)$ is positive-definite and uniformly nondegenerate, there exists $m$ such that $a_{11}a_{22}-a_{12}^2>m$. Due to  the continuity of $a(x)$, it's bounded on any compact set, so there exists $\gamma\in(0,1)$ such that $|a_{12}(x)|<\sqrt{a_{11}(x)a_{22}(x)-m}<(1-\gamma)\sqrt{a_{11}(x)a_{22}(x)}$ for every $x\in B(\bm{\mathrm x}_0,1)$.
Let $r$ be smaller than $1$ and such that $b_{12}(x)^2<(1+\gamma)^2b_{11}(x)b_{22}(x)$ for every $x\in B(\bm{\mathrm x}_0,r)$. Therefore, 
$$|a_{12}(x)b_{12}(x)|<(1-\gamma^2)\sqrt{a_{11}(x)a_{22}(x)b_{11}(x)b_{22}(x)}.$$It follows that
$$\sum_{i,j=1}^2 a_{ij}(x)b_{ij}(x)>0,$$which completes the proof.
\end{proof}

\begin{lemma}
Suppose that ${\bm{\varphi}}_t=(i,\varphi_t)\in\textbf{C}\left([0,T], I_i\right)$ is non-constant with $S({\bm{\varphi}})$ finite, and ${\bm{\varphi}}_0$ is at an exterior vertex $O$. If $t_0=\inf\{t:{\bm{\varphi}}_t\not=O\}$, then $\dot\varphi_{t_0+}=0$.
\label{zeroderivative}
\end{lemma}
\begin{proof}
Without loss of generality, we assume that $t_0=0$ and $\varphi_{0}=H(\bm{\mathrm x}_0)=0$, where $\bm{\mathrm x}_0$ is the origin, and a local minimum is achieved there. Due to the positive-definiteness of the Hessian matrix at $\bm{\mathrm x}_0$, we have $r,\lambda,\Lambda>0$ such that, for every $x\in B(\bm{\mathrm x}_0,r)$, we have the estimates: $H(x)\geq\lambda |x|^2$ and $|\nabla H(x)|\leq\Lambda |x|$. It follows that there exists $H_0>0$ such that, for every $H<H_0$, $C_i(H)\subseteq B(\bm{\mathrm x}_0,r)$. Furthermore, let the eigenvalues of $\sigma\sigma^*$ be bounded from above by M.
Thus, for $H<H_0$, $$B_i^2(H)\leq\max_{x\in C_i(H)}|\nabla H(x)^*\sigma(x)|^2\leq\max_{x\in C_i(H)}M|\nabla H(x)|^2\leq\max_{x\in C_i(H)}M\Lambda^2|x|^2,$$$$ H\geq\max_{x\in C_i(H)}\lambda|x|^2$$
$$\Rightarrow \frac{B_i^2(H)}{H}\leq\frac{M\Lambda^2}{\lambda}.$$
A direct consequence of this estimate is that if $t$ is so small that for every $ s\leq t$, $\varphi_s<H_0$, then
$$
    \int_0^t B_i^2(\varphi_s)ds\leq\int_0^t\frac{B_i^2(\varphi_s)}{\varphi_s}\cdot{\varphi_s} ds
    \leq\frac{M\Lambda^2}{\lambda}\cdot\int_0^t\varphi_s ds.
$$
By the Cauchy-Schwarz Inequality, when $t$ is sufficiently small,
\begin{align*}
    \left(\int_0^t|\dot\varphi_s| ds\right)^2&\leq\int_0^t\frac{{\dot\varphi_s}^2}{B_i^2(\varphi_s)}ds\cdot\int_0^t B_i^2(\varphi_s)ds\\
    &\leq\int_0^t\frac{{\dot\varphi_s}^2}{B_i^2(\varphi_s)}ds\cdot\frac{M\Lambda^2}{\lambda}\cdot\int_0^t\varphi_sds\\
    &\leq\int_0^t\frac{{\dot\varphi_s}^2}{B_i^2(\varphi_s)}ds\cdot\frac{M\Lambda^2}{\lambda}\cdot\int_0^t\int_0^s|\dot\varphi_u|duds\\
    &\leq\int_0^t\frac{{\dot\varphi_s}^2}{B_i^2(\varphi_s)}ds\cdot\frac{M\Lambda^2}{\lambda}\cdot t\int_0^t|\dot\varphi_s| ds.
\end{align*}
Thus we can conclude that $$\frac{\varphi_t}{t}\leq\frac{\int_0^t|\dot\varphi_s| ds}{t}\leq\frac{M\Lambda^2}{\lambda}\int_0^t\frac{{\dot\varphi_s}^2}{B_i^2(\varphi_s)}ds\to0,$$
as $t\to0$, which completes the proof.
\end{proof}

\textbf{Part 4.} Now we turn to the technical preparation for Section 3.3.
\begin{proof}[Verification of \eqref{derv1} and \eqref{derv2}]
\input{derivatives}
\end{proof}

\begin{proof}[Proof of Lemma~\ref{brownianatsaddle}]
    We only give the proof of the first statement here, and the others can be proved similarly. Our approach here is based on the reflection principle. Let us start with the probability of a larger event,
    \begin{align*}
        &{\bm{\mathrm P}}\left(\e^{\beta/2}{W}_{T}<-4\e^{(a-d)\beta}(a-d)\beta|\log\e|\right)\\
        \geq\ &{\bm{\mathrm P}}\left(\e^{\beta/2}{W}_{T}\in(-\e^{a\beta}-4\e^{(a-d)\beta}(a-d)\beta|\log\e|,-4\e^{(a-d)\beta}(a-d)\beta|\log\e|)\right)\\ 
        \geq\ &\e^{a\beta}\frac{1}{\sqrt{2T\pi\e^\beta}}\exp\left(-\frac{(4\e^{(a-d)\beta}(a-d)\beta|\log\e|+\e^{a\beta})^2}{2T\e^\beta}\right),
    \end{align*}
    while the probability of the difference is estimated, with  $\Tilde\e:=4\e^{(a-d)\beta}(a-d)\beta|\log\e|+\frac{1}{2}a\beta|\log\e|\e^{a\beta}$, as follows:
    \begin{align*}
        &{\bm{\mathrm P}}\left(\e^{\beta/2}{W}_{T}<-4\e^{(a-d)\beta}(a-d)\beta|\log\e|,\ \e^{\beta/2} W_t\geq\frac{1}{4}a\beta|\log\e|\e^{a\beta},\ \text{for some }t\in(0,T)\right)\\
        =\ &{\bm{\mathrm P}}\left(\e^{\beta/2}{W}_{T}>4\e^{(a-d)\beta}(a-d)\beta|\log\e|+\frac{1}{2}a\beta|\log\e|\e^{a\beta}\right)\\
        =\ &\frac{1}{\sqrt{2\pi\cdot \e^\beta T}}\int_{\Tilde{\e}}^\infty \exp({-\frac{x^2}{2\e^\beta T}})dx\\
        \leq\ &\frac{1}{\sqrt{2\pi\cdot \e^\beta T}}\cdot\frac{\e^\beta T}{\Tilde{\e}}\cdot\exp\left(-\frac{\Tilde{\e}^2}{2T\e^\beta}\right),
    \end{align*}where the first equality follows from the reflection principle and the inequality follows from the relation:$\int_y^\infty\exp(-cx^2)dx\leq\frac{1}{2cy}\exp(-cy^2)$, for all $y,c>0$.
    Consider the ratio of these two probabilities:
    \begin{align*}
        &\frac{{\bm{\mathrm P}}\left(\e^{\beta/2}{W}_{T}<-4\e^{(a-d)\beta}(a-d)\beta|\log\e|,\ \e^{\beta/2} W_t\geq\frac{1}{4}a\beta|\log\e|\e^{a\beta},\ \text{for some }t\in(0,T)\right)}{{\bm{\mathrm P}}\left(\e^{\beta/2}{W}_{T}<-4\e^{(a-d)\beta}(a-d)\beta|\log\e|\right)}\\
        \leq\ &\frac{\e^\beta T}{\e^{a\beta}\Tilde{\e}}\exp\left(\frac{(4\e^{(a-d)\beta}(a-d)\beta|\log\e|+\e^{a\beta})^2}{2T\e^\beta}-\frac{\Tilde{\e}^2}{2T\e^\beta}\right)\to 0,\ \text{as }\e\to0.
    \end{align*}
    Combining the estimates above, we obtain
    \begin{align*}
        &{\bm{\mathrm P}}\left(\e^{\beta/2}{W}_{T}<-4\e^{(a-d)\beta}(a-d)\beta|\log\e|,\ \e^{\beta/2} W_t<\frac{1}{4}a\beta|\log\e|\e^{a\beta},\ \forall\ 0<t<T\right)\\
        \geq\ &\frac{1}{2}\e^{a\beta}\frac{1}{\sqrt{2T\pi\e^\beta}}\exp\left(-\frac{(4\e^{(a-d)\beta}(a-d)\beta|\log\e|+\e^{a\beta})^2}{2T\e^\beta}\right)\\
    \geq\ &2\exp\left(-\e^{-(1-2(a-d))\beta-\kappa}\right),
    \end{align*}for $\e$ small enough.
\end{proof}

\begin{proof}[Proof of Lemma~\ref{escapetime}]
Let us look at the slow motion for small $\e$:
\begin{align*}
    dx_t&=\nabla^\perp H(x_t)dt;\\
    d\Tilde X_t^\e&=\nabla^{\perp}H(\Tilde X_t^\e)dt+\sqrt\e\sigma(\Tilde X_t^\e)d\Tilde{W_t}.
\end{align*}
 Let $F$ be the region between the level curves $\{x:H(x)=\frac{1}{2}\e^{a\beta}\}$ and $\{x:H(x)=\e^{(a-d)\beta}\}$, and between $\gamma_{out}$ and $\gamma_{in}$, where $x_t$ moves from $\gamma_{in}$ to $\gamma_{out}$. 
 The integrand is at least $c|\log\e|^2$ large (for some constant $c$) if the process is away from the critical point, and we would like to show that, with high probability,  the process does not spend more than $C|\log\e|$ time (for some constant $C$) in $F$ during a typical rotation. By \eqref{eq:T}, the deterministic process starting at any point in $F$ spends no more than time $C|\log\e|$ in $F$ before exit. We'll bound the exit time for the random process by approximating it with piecewise deterministic motion and will see that the small diffusion added does not significantly delay the time for the process to exit from $F$.

Define $T_1(x)=\inf\{t:x_t\not\in F,\ x_0=x\}$ and $T_2(x)=\inf\{t:x_t\in\gamma_{in},\ x_0=x\}$. For the process starting at $x$, define stopping times $\eta^x=\inf\{t:\Tilde X_t^\e\in C_{ki}(\e^{a\beta}/2)\cup C_{ki}(\e^{(a-d)\beta})\}$, $\tau_0^x=\inf\{t:\Tilde X_t^\e\in\gamma_{out}\}$, $\tau_1^x=\inf\{t:\Tilde X_t^\e\in\gamma_{in}\}$, $\tau_2^x=\inf\{t>\tau_1^x:\Tilde X_t^\e\in\gamma_{out}\}$, and $\hat\tau_j^x=\tau_j^x\wedge\eta^x$, $j=0,1,2$.
Since $\sigma$ is bounded, it's not hard to see that, for each $x\in F$ and $\Tilde X^\e$ starting at $x$, $${\bm{\mathrm P}}\left(\sup_{\Delta\in[0,1]}|\sqrt{\e}\int_t^{t+\Delta}\sigma(\Tilde X_s^\e)dW_s|>\e^{\frac{1}{2}-\frac{1}{2}(1-2(a-d))\beta-2\kappa}\right)<\exp\left(-\e^{-(1-2(a-d))\beta-4\kappa}\right).$$
We claim that, with $\overline{M}$ defined at the beginning of the Section 3.3 and $x\in\gamma_{in}$,
$$E:=\bigcap_{k=0}^{\overline M|\log\e|-1}\left\{\sup_{\Delta\in[0,1]}|\sqrt{\e}\int_k^{k+\Delta}\sigma(\Tilde X_s^\e)dW_s|\leq\e^{\frac{1}{2}-\frac{1}{2}(1-2(a-d))\beta-2\kappa}\right\}\subseteq\left\{\hat\tau_0^x\leq\overline M|\log\e|\right\}.$$
To verify this, note that, by Gronwall's inequality, the event $E$ is contained in
$$\bigcap_{k=0}^{\overline M|\log\e|-1}\left\{\sup_{\Delta\in[0,1]}|\Tilde X^\e_{k+\Delta}-x^k_{k+\Delta}|\leq e^L\cdot\e^{\frac{1}{2}-\frac{1}{2}(1-2(a-d))\beta-2\kappa}\right\},$$where $L$ is the Lipschitz constant of $\nabla H$ and $x^k$ is the deterministic process with random starting point defined by 
$$dx_t^k=\nabla^\perp H(x_t^k)dt,\ k\leq t\leq k+1,\ x_k^k=\Tilde X^\e_k.$$ Then, by the recalling that $\|J_{\psi^{-1}}\|<\overline M$ and $|\nabla T(\psi^{-1}(x))|\leq\frac{C}{|H(x)|}$, it follows that \begin{align}
    E&\subseteq\bigcap_{k=0}^{\overline M|\log\e|-1}\left\{\sup_{\Delta\in[0,1]}|\Tilde X^\e_{k+\Delta}-x^k_{k+\Delta}|\leq e^L\cdot\e^{\frac{1}{2}-\frac{1}{2}(1-2(a-d))\beta-2\kappa}\right\}\nonumber\\
    &\subseteq\bigcap_{k=0}^{\overline M|\log\e|-1}\left\{T_1(\Tilde X_{k+1}^\e)\leq (T_1(\Tilde X_{k}^\e)-1/2)\vee0\right\}\bigcup\left\{\hat\tau_0^x\leq\overline M|\log\e|\right\}\label{eq:exittime1}\\
    &\subseteq \ \left\{\hat\tau_0^x\leq\overline M|\log\e|\right\}.\label{eq:exittime2}
\end{align}
Here we need to explain why the last two inclusions hold.
\begin{enumerate}
\item[(1)] Let us start with \eqref{eq:exittime1}.
For each $0\leq k\leq\overline M|\log\e|-1$, if $\sup_{\Delta\in[0,1]}|\Tilde X^\e_{k+\Delta}-x^k_{k+\Delta}|\leq e^L\cdot\e^{\frac{1}{2}-\frac{1}{2}(1-2(a-d))\beta-2\kappa}$, then one of the following must happen:
\begin{enumerate}
\item If $\Tilde X^\e_{k+1}\not\in F$, then $\hat\tau_0^x\leq\overline{M}|\log\e|$.
\item If $\Tilde X^\e_{k+1}\in F$ and $x_{k+1}^k\in F$, then $T_1(x^k_{k+1})=T_1(\Tilde X_k^\e)-1$ and $|T_1(\Tilde X^\e_{k+1})-T_1(x^k_{k+1})|=|T(\psi^{-1}(\Tilde X^\e_{k+1}))-T(\psi^{-1}(x^k_{k+1}))|=o(1)$, so $T_1(\Tilde X^\e_{k+1})\leq(T_1(\Tilde X_{k}^\e)-1/2)$.
\item If $\Tilde X^\e_{k+1}\in F$, $x_{k+1}^k\not\in F$, and $H(x_{k+1}^k)\not\in[\frac{1}{2}\e^{a\beta},\e^{(a-d)\beta}]$, then $\Tilde X_k^\e\not\in F$ and $\hat\tau_0^x\leq\overline{M}|\log\e|$.
\item If $\Tilde X^\e_{k+1}\in F$, $x_{k+1}^k\not\in F$, $H(x_{k+1}^k)\in[\frac{1}{2}\e^{a\beta},\e^{(a-d)\beta}]$, and $x_{k+1}^k\in V_0\setminus V$, then, as in (b), we still have that $T_1(\Tilde X^\e_{k+1})\leq(T_1(\Tilde X_{k}^\e)-1/2)$.
\item If $\Tilde X^\e_{k+1}\in F$, $x_{k+1}^k\not\in F$, $H(x_{k+1}^k)\in[\frac{1}{2}\e^{a\beta},\e^{(a-d)\beta}]$, and $x_{k+1}^k\not\in V_0$, then we have a contradiction here, since $\psi$ is a diffeomorphism, $|\Tilde X^\e_{k+1}-x_{k+1}^k|\leq e^L\cdot\e^{\frac{1}{2}-\frac{1}{2}(1-2(a-d))\beta-2\kappa}$, and there is a constant distance between $\partial U$ and $\partial U_0$.
\end{enumerate}
\item[(2)] The last inclusion \eqref{eq:exittime2} is justified by the fact that $T_1(x)<\overline{M}|\log H(x)|<\frac{1}{2}\overline{M}|\log\e|-1$ on $F$ for $\e$ small enough. 

\end{enumerate}
Thus, by the strong Markov property,
\begin{align*}
    {\bm{\mathrm P}}(\hat\tau_0^x\leq\overline M|\log\e|)\geq {\bm{\mathrm P}}(E)&\geq \left(1-\exp\left(-\e^{-(1-2(a-d))\beta-4\kappa}\right)\right)^{\overline M|\log\e|}\\
    &\geq 1-\overline M|\log\e|\exp\left(-\e^{-(1-2(a-d))\beta-4\kappa}\right)\\
    &\geq 1-\exp\left(-\e^{-(1-2(a-d))\beta-2\kappa}\right).
\end{align*}

Now we look at one full rotation. For each $x\in\gamma_{out}\cap D_i(\frac{1}{2}\e^{a\beta},\e^{(a-d)\beta})$ and deterministic and random processes starting at $x$, consider the events 
\begin{align*}
    E_1&=\{\eta^x\leq\tau_1^x\}\cap\{\sup_{t\in[0,2T_2(x)]}|\Tilde X_t^\e-x_t|<\e^{\frac{1}{2}-\frac{1}{2}(1-2(a-d))\beta-2\kappa}\},\\
    E_2&=\{\eta^x>\tau_1^x\}\cap\{\sup_{t\in[0,2T_2(x)]}|\Tilde X_t^\e-x_t|<\e^{\frac{1}{2}-\frac{1}{2}(1-2(a-d))\beta-2\kappa}\}\cap\{\hat\tau_2^x-\hat\tau_1^x\leq\overline{M}|\log\e|\},\\
    E_3&=\{\eta^x>\tau_1^x\}\cap\{\sup_{t\in[0,2T_2(x)]}|\Tilde X_t^\e-x_t|<\e^{\frac{1}{2}-\frac{1}{2}(1-2(a-d))\beta-2\kappa}\}\cap\{\hat\tau_2^x-\hat\tau_1^x>\overline{M}|\log\e|\}.
\end{align*}
By the strong Markov property of $\Tilde X_t^\e$, ${\bm{\mathrm P}}(E_3)\leq\exp\left(-\e^{-(1-2(a-d))\beta-2\kappa}\right)$. Furthermore, note that 
\begin{align*}
    {\bm{\mathrm P}}(E_1\sqcup E_2\sqcup E_3)&={\bm{\mathrm P}}\left(\sup_{t\in[0,2T_2(x)]}|\Tilde X_t^\e-x_t|<\e^{\frac{1}{2}-\frac{1}{2}(1-2(a-d))\beta-2\kappa}\right)\\
    &>1-\exp\left(-\e^{-(1-2(a-d))\beta-4\kappa}\right),
\end{align*}
where last inequality follows from the fact that $T_2(x)$ is bounded by a constant on $\gamma_{out}$. We deduce that 
\begin{equation}
\label{eq:onerotation}
    {\bm{\mathrm P}}(E_1\cup E_2)\geq1-2\exp\left(-\e^{-(1-2(a-d))\beta-2\kappa}\right).
\end{equation}Observe that, on the event $E_1\cup E_2$, $\int_0^t|\log H(\Tilde X_s^\e)\nabla H(\Tilde X_s^\e)^*\sigma(\Tilde X_s^\e)|^2ds>t$ holds for all $t\leq\hat\tau_2^x$. Using the strong Markov property of $X_t^\e=\Tilde X_{t\e^{\beta-1}}^\e$ and \eqref{eq:onerotation} for every rotation prior to $\eta\wedge T$ (here the term ``rotation'' means that the process travels from $\gamma_{out}$ to $\gamma_{in}$ and then back to $\gamma_{out}$), we obtain the desired result by noting that the number of rotations $N$ is then bounded by $C\e^{\beta-1}$ provided that the analogue of $E_1\cup E_2$ holds for every rotation,
\begin{align*}
    &{\bm{\mathrm P}}\left(\int_0^t|\log H(X_s^\e)\rcoe|^2ds>t,\ \forall\ t\leq\eta\wedge T\right)\\
    \geq&\left(1-2\exp\left(-\e^{-(1-2(a-d))\beta-2\kappa}\right)\right)^N\\
    \geq&\ 1-2N\exp\left(-\e^{-(1-2(a-d))\beta-2\kappa}\right)\\
    \geq&\ 1-\exp\left(-\e^{-(1-2(a-d))\beta-\kappa}\right). 
\end{align*}
\end{proof}

\textbf{Part 5.} The last lemma is needed for the proof in Section 4.1.
\begin{lemma}
With the same assumptions as in Lemma \ref{condition3}, we have, 
for every $0\leq k<n$,
$${\bm{\mathrm E}}\left(\chi_{B_k\cap A_k}\cdot\exp\left(\frac{1-\mu}{2\e^\beta}\cdot\frac{n}{B^2\left(H(X_\kk^\e)\right)-\frac{L\delta'}{2}}\left|H(X_\kl^\e)-H(X_\kk^\e)\right|^2\right)\given\right)\leq C_\mu<\infty,$$
where $C_\mu$ is a constant that depends solely on $\mu$.
\label{finiteexpectation}
\end{lemma}
\begin{proof}
Without loss of generality, we can assume that $k=0$, the initial point is $x\in C$, and drop the conditioning on $\mathcal F_\kk$.
Note that
\begin{equation}
\label{eq:A.9}
    \lvert H(X^\e_{1/n})-H(x)\rvert\leq\e^\beta\left|\int_0^{1/n}AH(X_s^\e)ds\right|+\e^{\beta/2}\left|\int_0^{{1/n}}\rcoe dW_s\right|.
\end{equation}The first term on the right hand side is negligible compared to the second term, which needs to be studied further. Note that, on the event $B_0$, $H(X_t^\e)\in H(C)$ for $0\leq t\leq1/n$. For each $0\leq p\leq\lfloor\frac{1}{nT^\e_0}\rfloor$ and $t\geq pT^\e_0$, define the deterministic process with the random starting point:
$$d\xi_t^\e=\e^{\beta-1}\nabla^\perp H(\xi_t^\e)dt,\ \ \ \ \xi_{pT^\e_0}^\e=X^\e_{pT^\e_0}.$$
Then, on the event $A_0\cap B_0$,
\begin{align*}
    \lvert X_t^\e-\xi_t^\e\rvert\leq\ &\e^{\beta-1}\int_{pT^\e_0}^t\left|\nabla H^\perp(X_s^\e)-\nabla H^\perp(\xi_s^\e)\right|ds+\e^{\beta/2}\left|\int_{pT^\e_0}^t\sigma(X_s^\e)dW_s\right|\\
    \leq\ &\e^{\beta-1}\int_{pT^\e_0}^tL\left|X_s^\e-\xi_s^\e\right|ds+M\cdot\e^{\frac{1-\beta}{4}}.
\end{align*}
So, for $\e$ small enough, by Gronwall's inequality,
\begin{equation}
\label{deviation}
    \lvert X_t^\e-\xi_t^\e\rvert\leq M\exp(L\cdot T(H(x)))\cdot\e^{\frac{1-\beta}{4}}\leq M\exp(LM)\cdot\e^{\frac{1-\beta}{4}}\leq\frac{\delta'}{M}.
\end{equation}
Using \eqref{deviation} and Lipschitz continuity and boundedness of $\lvert\nabla H(x)^*\sigma(x)\rvert^2$, $B^2(H)$, and $T(H)$, we estimate $|\rcoe|^2$ in terms of $B^2(X_s^\e)$ on the interval $[pT^\e_0,(p+1)T^\e_0]$, on the event $A_0\cap B_0$:
\begin{align*}
    &\int_{pT^\e_0}^{(p+1)T^\e_0}\lvert\rcoe\rvert^2ds\\ 
    \leq &\int_{pT^\e_0}^{(p+1)T^\e_0}\lvert\nabla H(\xi_s^\e)^*\sigma(\xi_s^\e)\rvert^2ds+\e^{1-\beta}T(H(x))L\cdot\frac{\delta'}{M}\\
    \leq&\int_{pT^\e_0}^{pT^\e_0+\e^{1-\beta}T(H(X_{pT^\e_0}^\e))}\lvert\nabla H(\xi_s^\e)^*\sigma(\xi_s^\e)\rvert^2ds+\e^{1-\beta}\cdot L\delta'M+\e^{1-\beta}T(H(x))L\cdot\frac{\delta'}{M}\\
    \leq&\ \e^{1-\beta}T(H(X_{pT^\e_0}^\e))B^2(H(X_{pT^\e_0}^\e))+\e^{1-\beta}\cdot L\delta'(M+1)\\
    \leq&\ T^\e_0B^2(H(X_{pT^\e_0}^\e))+\e^{1-\beta}\cdot L\delta'(2M+1).
\end{align*}This is valid for every $0\leq p\leq\lfloor\frac{1}{nT^\e_0}\rfloor$. Combining the contributions from all the time intervals, since $B^2(H)$ is Lipschitz and $\delta'$ is chosen to be small enough, it follows that on the event $A_0\cap B_0$:
\begin{align*}
    &\int_0^{1/n}\lvert\rcoe\rvert^2ds\\
    \leq\ &\sum_{p=0}^{\lfloor\frac{1}{nT^\e_0}\rfloor-1}\int_{pT^\e_0}^{(p+1)T^\e_0}\lvert\rcoe\rvert^2ds+T^\e_0 M\\
    \leq&\ \sum_{p=0}^{\lfloor\frac{1}{nT^\e_0}\rfloor-1}\left(T^\e_0B^2(H(X_{pT^\e_0}^\e))+\e^{1-\beta}\cdot L\delta'(2M+1)\right)+T^\e_0 M\\
    \leq&\ \frac{1+\mu}{n}\left[B^2(H(x))+\frac{L\delta'}{2}\right],
\end{align*}
for $\e$ sufficiently small.
Let $Q=\frac{1+\mu}{n}\left[B^2(H(x))+\frac{L\delta'}{2}\right]$. Since $|AH(x)|\leq M$ in $C$ and \eqref{eq:A.9} holds,
\begin{align*}
    &{\bm{\mathrm E}}\left(\chi_{B_0\cap A_0}\cdot\exp\left(\frac{1-\mu}{2\e^\beta}\cdot\frac{n}{B^2\left(H(x)\right)-\frac{L\delta'}{2}}\left|H(X_{1/n}^\e)-H(x)\right|^2\right)\right)\\
    \leq\ &{\bm{\mathrm E}}\left(\chi_{B_0\cap A_0}\cdot\exp\left(\frac{n(1-\mu)}{2B^2\left(H(x)\right)-L\delta'}\left(\left|\int_0^{\frac{1}{n}}\rcoe dW_s\right|+\frac{M}{n}\e^{\beta/2}\right)^2\right)\right)\\
    =\ &{\bm{\mathrm E}}\left(\chi_{B_0\cap A_0}\cdot\exp\left(\frac{n(1-\mu)}{2B^2\left(H(x)\right)-L\delta'}\left(\left|\Tilde W\left(\int_0^{\frac{1}{n}}|\rcoe|^2ds\right)\right|+\frac{M}{n}\e^{\beta/2}\right)^2\right)\right)\\
    \leq\ &{\bm{\mathrm E}}\left(\chi_{B_0\cap A_0}\cdot\exp\left(\frac{1-\mu}{2}\cdot\frac{n}{B^2\left(H(x)\right)-\frac{L\delta'}{2}}\cdot\left(\sup_{0\leq t\leq Q} |\Tilde W(t)|+\frac{M}{n}\e^{\beta/2}\right)^2\right)\right)\\
    \leq\ &{\bm{\mathrm E}}\left(\exp\left(\frac{1-\mu^2}{2}\cdot\frac{B^2\left(H(x)\right)+\frac{L\delta'}{2}}{B^2\left(H(x)\right)-\frac{L\delta'}{2}}\cdot\frac{1}{Q}\left(\sup_{0\leq t\leq Q} |\Tilde W(t)|+\frac{M}{n}\e^{\beta/2}\right)^2\right)\right)\\
    \leq\ &{\bm{\mathrm E}}\left(\exp\left(\frac{1-\mu^2/2}{2}\cdot\frac{1}{Q}\left(\sup_{0\leq t\leq Q} |\Tilde W(t)|+\frac{M}{n}\e^{\beta/2}\right)^2\right)\right)\\
    \leq\ &{\bm{\mathrm E}}\left(\exp\left(\frac{1-\mu^2/3}{2}\cdot\frac{1}{Q}\sup_{0\leq t\leq Q} |\Tilde W(t)|^2\right)\right)+\exp\left(\frac{1-\mu^2/2}{2}\right),
\end{align*}
where the last inequality follows, for $\e$ small, by integrating over the sets $\{\sup_{0\leq t\leq  Q} |\Tilde W(t)|>\sqrt Q/2\}$ and $\{\sup_{0\leq t\leq Q} |\Tilde W(t)|\leq \sqrt Q/2\}$ separately.
For every $a>0$, by the reflection principle,
\begin{align*}
    {\bm{\mathrm P}}\left(\sup_{0\leq t\leq Q} |\Tilde W(t)|^2>a^2\right)&\leq4{\bm{\mathrm P}}\left(\Tilde W_{Q}>a\right)=2{\bm{\mathrm P}}\left(|\Tilde W_{Q}|^2>a^2\right).
\end{align*}
Therefore, for every $a>0$, ${\bm{\mathrm P}}(\xi>a)\leq 2{\bm{\mathrm P}}(\eta>a)$, where
$$\xi=\exp\left(\frac{1-\mu^2/3}{2}\cdot\frac{1}{Q}\sup_{0\leq t\leq Q} |\Tilde W(t)|^2\right)>0,$$
$$\eta=\exp\left(\frac{1-\mu^2/3}{2}\cdot\frac{1}{Q}|\Tilde W(Q)|^2\right)>0.$$
Since $${\bm{\mathrm E}}\left(\eta\right)={\bm{\mathrm E}}\left(\exp\left(\frac{1-\mu^2/3}{2}\cdot\frac{1}{Q}|\Tilde W(Q)|^2\right)\right):=\Tilde c_\mu<\infty,$$we have
$${\bm{\mathrm E}}\left(\xi\right)=\int_0^\infty {\bm{\mathrm P}}\left(\xi>x\right)dx\leq2\int_0^\infty {\bm{\mathrm P}}\left(\eta>x\right)dx=2\Tilde c_\mu<\infty,$$
which completes the proof of the Lemma.
\end{proof}

%% file: compactness.tex
\begin{lemma}
\label{zeroder}
Suppose that $\varphi:[0,T]\to\mathbb R$ is absolutely continuous, and let $E=\{t:\varphi_t=c\}$, where $c$ is a fixed constant. Then, for almost every point $t$ in $E$ (with respect to the Lebesgue measure), we have $\dot\varphi_t=0$.
\end{lemma}
\begin{proof}
Since $\varphi$ is absolutely continuous on $[0,T]$, it's differentiable almost everywhere. 
Let $F$ be the subset of $E$ where $\varphi$ is not differentiable, so $\lambda(F)=0$. 
For each $t_0\in E\setminus F$, the limit $$\lim_{h\to0}\frac{\varphi_{t_0+h}-\varphi_{t_0}}{h}$$ exists. If $\dot\varphi_{t_0}\not=0$, then $\varphi_t\not=\varphi_{t_0}=0$ on $[t_0-h,t_0+h]$ for some positive $h$, which means that $t_0$ is an isolated point in $E$. There can be at most countably many isolated points inside a set of finite measure, so $G:=\{t:\dot\varphi_t\text{ exists and is nonzero}\}$ has measure $0$. Thus $E\setminus(F\cup G)$ is the set where $\dot\varphi_t=0$, and $\lambda(F\cup G)=0$.
\end{proof}
\begin{lemma}
\label{semic}
The mapping $S:\textbf{C}\left([0,T],\Gamma\right)\to\mathbb R$ is lower semicontinuous.
\end{lemma}
\begin{proof}
Suppose that $\bm{\varphi}^{(l)}\to\bm{\varphi}$ in $\textbf{C}\left([0,T],\Gamma\right)$.
It is sufficient to consider the case where $\liminf_{l\to\infty}S(\bm{\varphi}^{(l)})$ is finite. 
For $l$ large enough, $\pl$ is uniformly bounded by some $M_1\in\mathbb R$. Consequently, $B_{i^{(l)}}^2(\pl)$ is uniformly bounded by some $M_2\in\mathbb R$.
Using the Lemma from \cite{nagy} on page 75, we evaluate
\begin{align*}
    &\sup_{0\leq t_0<... <t_N\leq T}\sum_{i=1}^N\frac{|\varphi_{t_i}-\varphi_{t_{i-1}}|^2}{t_i-t_{i-1}}\\
    =&\sup_{0\leq t_0< ... <t_N\leq T}\lim_{l\to\infty}\sum_{i=1}^N\frac{|\pl_{t_i}-\pl_{t_{i-1}}|^2}{t_i-t_{i-1}}\\
    \leq&\varliminf_{l\to\infty}\sup_{0\leq t_0< ... <t_N\leq T}\sum_{i=1}^N\frac{|\pl_{t_i}-\pl_{t_{i-1}}|^2}{t_i-t_{i-1}}\\
    =&\varliminf_{l\to\infty}\int_0^T|\dot\varphi^{(l)}|^2\\
    \leq&\ 2M_2\varliminf_{l\to\infty}S(\bm{\varphi}^{(l)})<\infty,
\end{align*}
which implies that $\varphi$ is absolutely continuous and has square integrable derivative. 
Now fix $\gamma>0$, and define $U_\gamma=\{t:B_{i_t}^2(\varphi_t)>\gamma\}$. This is an open set that is a union of at most countably many open intervals $\{\Tilde I_j:j\geq1\}$ inside $[0,T]$. 
Take $L$ such that $B_i^2(H)$ is Lipschitz continuous on $[\gamma,M_2]$ with constant $L$ uniformly in $i$. 
Then, by the uniform continuity of $\varphi$, for each $\e>0$, there exists $\delta>0$ such that $|\varphi_t-\varphi_s|<\gamma\e/L$ whenever $t,s\in \Tilde I_j$ and $|t-s|<\delta$. 
Notice that on a single interval $\Tilde I_j$, $i_t$ remains the same.
Hence $t,s\in \Tilde I_j$ and $|t-s|<\delta$ implies $$\frac{B_{i_s}^2(\varphi_s)}{B_{i_t}^2(\varphi_t)}\leq 1+\e.$$ 
Take $n$ large enough so that $T/n<\delta$, and, for each $\Tilde I_j=(a,b)$, let $a_k=a+\frac{k}{n}(b-a)$. Again, $i_t$ remains the same on $\Tilde I_j$, so
\begin{align*}
    \int_a^b\frac{|\dot\varphi_t|^2}{B_i^2(\varphi_t)}dt&\leq(1+\e)\sum_{k=0}^{n-1}\frac{1}{B^2_{i}(\varphi_{a_k})}\int_{a_k}^{a_{k+1}}|\dot\varphi_t|^2dt\\
    &=(1+\e)\sum_{k=0}^{n-1}\frac{1}{B^2_i(\varphi_{a_k})}\sup_{a_k\leq t^0_k<...<t^{M_k}_k\leq a_{k+1}}\sum_{i=1}^{M_k}\frac{|\varphi_{t_k^i}-\varphi_{t_k^{i-1}}|^2}{t_k^i-t_k^{i-1}}\\
    &=(1+\e)\sum_{k=0}^{n-1}\frac{1}{B_i^2(\varphi_{a_k})}\sup_{a_k\leq t^0_k<...<t^{M_k}_k\leq a_{k+1}}\lim_{l\to\infty}\sum_{i=1}^{M_k}\frac{|\pl_{t_k^i}-\pl_{t_k^{i-1}}|^2}{t_k^i-t_k^{i-1}}\\
    &\leq(1+\e)\sum_{k=0}^{n-1}\frac{1}{B^2_i(\varphi_{a_k})}\varliminf_{l\to\infty}\sup_{a_k\leq t^0_k<...<t^{M_k}_k\leq a_{k+1}}\sum_{i=1}^{M_k}\frac{|\pl_{t_k^i}-\pl_{t_k^{i-1}}|^2}{t_k^i-t_k^{i-1}}\\
    &=(1+\e)\sum_{k=0}^{n-1}\frac{1}{B_i^2(\varphi_{a_k})}\varliminf_{l\to\infty}\int_{a_k}^{a_{k+1}}|\dot\varphi_t^{(l)}|^2\\
    &\leq(1+\e)^2\sum_{k=0}^{n-1}\varliminf_{l\to\infty}\int_{a_k}^{a_{k+1}}\frac{|\dot\varphi_t^{(l)}|^2}{B_i^2(\pl_t)}dt\\
    &\leq(1+\e)^2\varliminf_{l\to\infty}\int_a^b\frac{|\dot\varphi_t^{(l)}|^2}{B_i^2(\pl_t)}dt.
\end{align*}
Since $\e$ was chosen arbitrarily, we obtain $\int_a^b\frac{|\dot\varphi_t|^2}{B_i^2(\varphi_t)}dt\leq\liminf_{l\to\infty}\int_a^b\frac{|\dot\varphi_t^{(l)}|^2}{B_i^2(\pl_t)}dt$, which implies that $\int_{U_\gamma}\frac{|\dot\varphi_t|^2}{B_{i_t}^2(\varphi_t)}dt\leq\liminf_{l\to\infty}\int_0^T\frac{|\dot\varphi_t^{(l)}|^2}{B_{i_t}^2(\pl_t)}dt$. 
Since this is true for every $\gamma>0$, we have $$\int_{U}\frac{|\dot\varphi_t|^2}{B_{i_t}^2(\varphi_t)}dt\leq\varliminf_{l\to\infty}\int_0^T\frac{|\dot\varphi_t^{(l)}|^2}{B_{i_t}^2(\pl_t)}dt,$$
where $U=\{t:B_{i_t}^2(\varphi_t)>0\}$. 
By convention $0/0=0$ and Lemma \ref{zeroder}, we have $$\int_{[0,  T]\setminus U}\frac{|\dot\varphi_t|^2}{B_{i_t}^2(\varphi_t)}dt=0,$$ which implies that $$S(\bm{\varphi})=\frac{1}{2}\int_{0}^T\frac{|\dot\varphi_t|^2}{B_{i_t}^2(\varphi_t)}dt\leq\varliminf_{l\to\infty}\frac{1}{2}\int_0^T\frac{|\dot\varphi_t^{(l)}|^2}{B_{i_t}^2(\pl_t)}dt=\varliminf_{l\to\infty}S(\bm{\varphi}^{(l)}).$$So $S$ is lower semicontinuous.
\end{proof}
\begin{lemma}
\label{precom}
The set $\Phi_\delta(s)=\left\{r(\bm{\varphi}_0,Y(x_0))\leq\delta:S(\bm{\varphi})\leq s\right\}$ is precompact in $\textbf{C}\left([0,T],\Gamma\right)$ for any $\delta\geq0$, $s\geq0$.
\end{lemma}
\begin{proof}
We first prove the uniform boundedness of $\Phi_\delta(s)$. 
Since $H$ has bounded second derivatives, $|\nabla H(x)|$ is Lipschitz continuous with coefficient $L>0$. 
By assumption 5, we can assume that $\lVert\sigma\rVert<M$.
And due to assumption 2, we can find $M_1>|x_0|$ such that $H(x)\geq A_1|x|^2$ for all $|x|\geq M_1-|x_0|$, $M_1>\frac{2|\nabla H(x_0)|}{L}$, and $L(M_1-|x_0|)>|\nabla H(0)|$.
Notice that
$$|H(x)|\leq H(x_0)+\frac{1}{2}L|x-x_0|^2+|\nabla H(x_0)|\cdot|x-x_0|.$$
Hence, $|H(x)|\geq|H(x_0)|+M_1^2L$ implies that $|x-x_0|\geq M_1$, and therefore $|H(x)|\geq A_1|x|^2$. 
Let $M_2=|H(x_0)|+\delta+1+M_1^2L+\frac{16sM^2TL^2}{A_1}$. 
Suppose that $\Phi_\delta(s)$ is not uniformly bounded, so there exists $\bm{\varphi}=(i_t,\varphi_t)\in\Phi_\delta(s)$  such that $\sup_{[0,T]}|\varphi|>2M_2$. Let $\tau=\inf\{t:|\varphi_t|=2M_2\}>0$ and $\tau_1=\sup\{t<\tau,|\varphi_t|=M_2\}>0$. 
For $t\in[\tau_1,\tau]$, $|\varphi_t|\geq M_2\geq|H(x_0)|+M_1^2L$, which implies that $|x-x_0|\geq M_1$ for all $x\in C_{i_t}(\varphi_t)$. 
Thus we obtain
$$B_{i_t}^2(\varphi_t)\leq\sup_{x\in C_{i_t}(\varphi_t)}|\nabla H^*(x)\sigma(x)|^2\leq M^2\sup_{x\in C_{i_t}(\varphi_t)}(L|x|+|\nabla H(0)|)^2\leq\frac{4M^2L^2}{A_1}|\varphi_t|.$$ 
Thus
\begin{align*}
    M_2^2&=|\varphi_\tau-\varphi_{\tau_1}|^2\\
    &=\left|\int_{\tau_1}^\tau\dot\varphi_tdt\right|^2\\
    &\leq\int_{\tau_1}^\tau\frac{|\dot\varphi_t|^2}{B_{i_t}^2(\varphi_t)}dt\cdot\int_{\tau_1}^\tau B_{i_t}^2(\varphi_t)dt\\
    &\leq 2s\cdot T\cdot\frac{4M^2L^2}{A_1}\cdot 2M_2\\
    &=\frac{16M^2sTL^2}{A_1}M_2.
\end{align*}
However, by definition, $M_2>\frac{16M^2sTL^2}{A_1}$, which leads to a contradiction. 
Next, we prove the equicontinuity of $\Phi_\delta(s)$. 
Since $\Phi_\delta(s)$ is uniformly bounded, we assume that $B_{i_t}^2(\varphi_t)\leq M_3$ for each $\varphi\in\Phi_\delta(s)$. 
Instantly, we obtain
$$r(\bm{\varphi}_{t+h},\bm{\varphi}_t)^2\leq\left(\int_t^{t+h}|\dot\varphi_t|dt\right)^2\leq\int_t^{t+h}\frac{|\dot\varphi_t|^2}{B_{i_t}^2(\varphi_t)}dt\cdot\int_t^{t+h}B_{i_t}^2(\varphi_t)dt\leq h\cdot 2sM_3.$$
Thus $$r(\bm{\varphi}_{t+h},\bm{\varphi}_t)\leq\sqrt{2sM_3}\cdot\sqrt{h}.$$
The precompactness of $\Phi_\delta(s)$ now follows from the Arzelà–Ascoli theorem.
\end{proof}
\begin{theorem}
\label{compactness}
The set $\Phi_\delta(s)=\left\{r(\bm{\varphi_0},Y(x_0))\leq\delta:S(\varphi)\leq s\right\}$ is compact in $\textbf{C}\left([0,T],\mathbb R\right)$ for any $\delta\geq0$, $s\geq0$.
\end{theorem}
\begin{proof}
This is a direct consequence of Lemmas \ref{semic} and \ref{precom}.
\end{proof}

%% file: derivatives.tex
\begin{align*}
    \ \lvert\frac{\partial T}{\partial\mu}(\mu,\nu)\rvert
    =&\left\lvert\frac{1}{2}\int_{\nu}^l\frac{\partial}{\partial\mu}\left(\frac{\mathrm {det}(J_\psi(\sqrt{y^2+\mu^2-\nu^2},y))}{\sqrt{y^2+\mu^2-\nu^2}}\right)dy\right\rvert\\
    \leq&\ \frac{\overline M}{2}\int_\nu^l\left(\frac{\mu}{y^2+\mu^2-\nu^2}+\frac{\mu}{(y^2+\mu^2-\nu^2)^{3/2}}\right)dy\\
    =&\ \frac{\overline M}{2}\left(\left[\arctan{\frac{y}{\sqrt{\mu^2-\nu^2}}}\cdot\frac{\mu}{\sqrt{\mu^2-\nu^2}}\right]_\nu^l+\left[\frac{y}{\sqrt{y^2+\mu^2-\nu^2}}\cdot\frac{\mu}{\mu^2-\nu^2}\right]^l_\nu\right)\\
    \leq&\ \overline M\left(\frac{\pi\mu}{\sqrt{\mu^2-\nu^2}}+\frac{\mu}{(\mu^2-\nu^2)\sqrt{1+\mu^2-\nu^2}}+\frac{|\nu|}{\mu^2-\nu^2}\right)\\
    \leq&\ \frac{C}{H(\psi(\mu,\nu))}.
\end{align*}
\begin{align*}
    \ \lvert\frac{\partial T}{\partial\nu}(\mu,\nu)\rvert
    =&\left\lvert\frac{1}{2}\int_{\nu}^l\frac{\partial}{\partial\nu}\left(\frac{\mathrm {det}(J_\psi(\sqrt{y^2+\mu^2-\nu^2},y))}{\sqrt{y^2+\mu^2-\nu^2}}\right)dy-\frac{\mathrm {det}(J_\psi(\mu,\nu))}{2\lvert\mu\rvert}\right\rvert\\
    \leq&\ \frac{\overline M}{2}\int_\nu^l\left(\frac{\lvert\nu\rvert}{y^2+\mu^2-\nu^2}+\frac{\lvert\nu\rvert}{(y^2+\mu^2-\nu^2)^{3/2}}\right)dy+\frac{\overline{M}}{2\mu}\\
    =&\ \frac{\overline M}{2}\left(\left[\arctan{\frac{y}{\sqrt{\mu^2-\nu^2}}}\cdot\frac{\lvert\nu\rvert}{\sqrt{\mu^2-\nu^2}}\right]_\nu^l+\left[\frac{y}{\sqrt{y^2+\mu^2-\nu^2}}\cdot\frac{\lvert\nu\rvert}{\mu^2-\nu^2}\right]^l_\nu+\frac{1}{\mu}\right)\\
    \leq&\ \overline M\left(\frac{|\pi\nu|}{\sqrt{\mu^2-\nu^2}}+\frac{|\nu|}{(\mu^2-\nu^2)\sqrt{1+\mu^2-\nu^2}}+\frac{\nu^2}{\mu(\mu^2-\nu^2)}+\frac{1}{\mu}\right)\\
    =&\ \overline M\left(\frac{|\pi\nu|}{\sqrt{\mu^2-\nu^2}}+\frac{|\nu|}{(\mu^2-\nu^2)\sqrt{1+\mu^2-\nu^2}}+\frac{\mu}{\mu^2-\nu^2}\right)\\
    \leq&\ \frac{C}{H(\psi(\mu,\nu))}.
\end{align*}

%% file: Acknowledgement.tex
\section*{Acknowledgements}
I am grateful to Prof. M. Freidlin and Prof. L. Koralov for helpful discussions and useful comments.